\documentclass{scrartcl}

\usepackage{iftex}

\PassOptionsToPackage{warnings-off={mathtools-colon,mathtools-overbracket}}{unicode-math}

\ifPDFTeX 
  \usepackage[utf8]{inputenc}
  \usepackage[T1]{fontenc}
  \usepackage{lmodern}
\else
  \usepackage[no-math]{fontspec}
\fi

\usepackage[british]{babel}
\usepackage{etoolbox,xifthen,xparse}

\usepackage[dvipsnames,svgnames]{xcolor}
\usepackage{aliascnt}
\usepackage[inline]{enumitem}

\usepackage{amsmath}
\usepackage{amsfonts,amssymb,amsthm,mathtools,booktabs}
\ifPDFTeX
  \usepackage{bbm}
  \usepackage{mathrsfs}
\else
  \usepackage{unicode-math}
  \let\mathbbm\mathbb
\fi
\usepackage[retainorgcmds]{IEEEtrantools}

\usepackage{url,doi}
\usepackage[numbers]{natbib}
\usepackage{subcaption}

\usepackage[capitalise]{cleveref}
\usepackage{hyperref,bookmark}

\hypersetup{hidelinks}

\newcommand{\newsstheorem}[2]{
  \newaliascnt{#1}{dummy}
  \newtheorem{#1}[#1]{#2}
  \aliascntresetthe{#1}
  \expandafter\def\csname #1autorefname\endcsname{#2}
}
\theoremstyle{plain}
\newsstheorem{theorem}{Theorem}
\newsstheorem{proposition}{Proposition}
\newsstheorem{corollary}{Corollary}
\newsstheorem{lemma}{Lemma}
\theoremstyle{definition}
\newsstheorem{definition}{Definition}
\newsstheorem{example}{Example}
\newsstheorem{assumption}{Assumption}
\theoremstyle{remark}
\newsstheorem{remark}{Remark}

\setlist[enumerate,1]{label={(\roman*)}}
\setlist[enumerate,2]{label={(\alph*)}}
\setlist[enumerate,3]{label={(\Roman*)}}

\makeatletter \newcommand\mathof[1]{{\operator@font#1}} \makeatother
\newcommand\dd{\mathof{d}}

\newcommand\defbb[1]{\expandafter\newcommand\csname #1#1\endcsname{\mathbb{#1}}}
\defbb{N} \defbb{Z} \defbb{Q} \defbb{R} \defbb{C}
\defbb{P} \defbb{E}

\newcommand\Ind{\mathbbm{1}}
\newcommand\Indic[1]{\Ind_{\{#1\}}}
\newcommand\from{\colon}

\newcommand\Aa{\mathcal{A}}
\newcommand\Ff{\mathscr{F}}

\providecommand{\email}[1]{\href{mailto:#1}{\nolinkurl{#1}}}

\usepackage[colorinlistoftodos]{todonotes}
\presetkeys{todonotes}{size=\scriptsize}{}

\newcommand\revised[1]{#1}

\title{Stochastic optimal control of L\'evy tax processes with bailouts}
\author{
  D.\ Al Ghanim\footnote{Department of Mathematics, College of Science, King Saud University, P.O.Box 2455, Riyadh 11451, Saudi Arabia}
  \and 
  R.\ Loeffen\footnote{University of Liverpool, Institute for Financial and Actuarial Mathematics, Mathematical Sciences building, Liverpool L69 7ZL, ronnie.loeffen@liverpool.ac.uk}
  \and
  A.\ R.\ Watson\footnote{University College London, London, UK, alexander.watson@ucl.ac.uk}
}
\date{\today}

\begin{document}
\maketitle

\begin{abstract}
  We consider controlling the paths of a spectrally negative Lévy process by
  two means: the subtraction of `taxes' when the process is at an all-time
  maximum, and the addition of `bailouts' which keep the value of the process
  above zero. We solve the corresponding stochastic optimal control problem of
  maximising the expected present value of the difference between taxes
  received and cost of bailouts given. Our class of taxation controls is larger
  than has been considered up till now in the literature and makes the problem
  truly two-dimensional rather than one-dimensional.  Along the way, we define
  and characterise a large class of controlled Lévy processes to which the
  optimal solution belongs, which extends a known result for perturbed Brownian
  motions to the case of a general L\'evy process with no positive jumps. 
\end{abstract}

\paragraph{Key words and phrases.} Risk process, tax process, spectrally negative L\'evy process, capital injections, optimal control, perturbed L\'evy process, Skorokhod reflection.

\section{Introduction}

Consider the capital of an insurance company as it evolves over time.  A
government considers two interventions: first, a loss carryforward taxation
regime, in which some proportion, to be referred to as the tax rate, of the
company's increase in capital is taken whenever the level of capital reaches a
new record; and second, a bailout system, in which the government injects money
to the company in order to keep its capital level positive. We assume the
government will bail the company out (instead of letting it go bankrupt); so the company is considered `too big to fail’.
\revised{Though one can debate the use of such bailouts, they do occur,
most notably during the 2008 financial crisis, and since they affect the tax
strategy it is important to seriously consider models with this feature.}
Assume  the government wants to choose the tax rate and the
size and timing of the bailouts in order to maximise the expected discounted
difference between tax revenue and cost of bailouts. In this work, we show
that, when the company's capital before government interventions is modelled by
a spectrally negative Lévy process and the tax rate is at least $\alpha\ge 0$
and at most $\beta\in(0,1)$, this can be achieved using a threshold tax rate and
minimal bailouts.  Informally, a threshold tax rate means that tax is paid at
rate $\alpha$ when the capital level is below a threshold $b\ge 0$, and at rate
$\beta$ when it is above.  By minimal bailouts, we mean that only the minimum
amount of capital is injected  in order to keep the company solvent.
\revised{It makes sense that such a strategy might be optimal: bailouts are costly and so should be minimal, if the capital is large it is relatively safe to impose the highest tax rate, whereas if the capital level is low, then the risk of having to do costly bailouts is high which justifies reducing the tax burden of the company as much as possible. Note that the minimum and maximum tax rate $\alpha$ and $\beta$ are fixed and cannot be changed by the government.}

There is some existing literature on optimal taxation of L\'evy tax processes.
In the setting where there are no government bailouts (so the company
ceases to exist when its capital drops below 0),  Albrecher et al.~in
\cite[Section~4]{LevyInsuranceRiskProcessWithTax}  
start with a threshold tax rate (with the minimum tax rate
being $\alpha=0$) and then look for the optimal threshold level $b$ that
maximises tax revenue (at any initial capital level);
see also \cite[Section~2.3]{LundbergRiskProcessWithTax} for the special case where the spectrally
negative L\'evy process is a compound Poisson risk process with exponentially
distributed claims.
Wang et al.\ \cite{wang2019optimalLastOne} do the same but either include minimal bailouts
(and their cost) or a terminal value at the time of bankruptcy.
Wang and Hu \cite{OptimalLossCarryForwardWang} maximised tax revenue without
bailouts over all $[\alpha,\beta]$-valued `latent' tax rate functions (i.e.,~a tax
rate which is a function of the running supremum of the uncontrolled capital
level; see \cite{ALW-tax-equiv} for this terminology).

A closely related optimal control problem is the optimal dividend problem with
mandatory capital injections (so bankruptcy of the company is not allowed)
where money can be taken out, in an adapted way,
to give as dividends to
shareholders who in return cover the capital injections. This problem where one
wants to maximise expected discounted value of the paid out dividends minus the
cost of capital injections has been studied in \cite{shreveetal1984},
respectively \cite{APP2008}, in the case where the uncontrolled risk process is
a diffusion, respectively spectrally negative L\'evy process. In the latter
case it was shown in \cite{APP2008} (see also  \cite[Example 1]{shreveetal1984}
for the case of a Brownian motion with drift) that the optimal strategy is a
dividend barrier strategy where dividends are paid out in a minimal way to keep
the capital of the company below a certain level in combination with injecting
capital in a minimal way to keep it positive. This mirrors our main result for
the optimal taxation problem with mandatory bailouts. 

We remark that in optimal
dividend problems where no capital injections are allowed, or where capital
injections are optional rather than mandatory, the optimal dividend strategy
can be more exotic and one needs to assume a condition on the 
L\'evy process in order for the dividend barrier strategy to be
optimal, see \cite{loeffen2008} and \cite{gajek_kucinski}. This is consistent
with problems involving optimal taxation without bailout studied in the
aforementioned papers
\cite{LevyInsuranceRiskProcessWithTax,OptimalLossCarryForwardWang,wang2019optimalLastOne}
where results on optimality of the threshold tax rate strategy are provided
under some conditions on the L\'evy process.
Indeed, such conditions are not required in
\cite{wang2019optimalLastOne} for the problem where minimal bailouts are present.

We now rigorously state our stochastic optimal control problem of interest. For any stochastic
process $Y=(Y_t)_{t\geq 0}$ whose sample paths has right and left limits, we write $Y_{t-}=\lim_{s\uparrow t}Y_s$,
$Y_{t+}=\lim_{s\downarrow t} Y_s$ and $\Delta Y_s:=Y_{s+} - Y_{s-}$. Further, for a measurable function $f$  and an
increasing and right-continuous function $G$, we define $\int_{0^+}^t f(s) \dd G(s)$ as the Stieltjes integral
over $(0,t]$, and $\int_{0}^t f(s) \dd G(s)$ as the Stieltjes integral over
$[0,t]$; that is, $\int_{0}^t f(s) \dd G(s) = f(0)G(0) + \int_{0^+}^t f(s) \dd
G(s)$. For reasons that will become clear, we need to, for a given path, define a running supremum with a prescribed starting value. To this end, for a path $Y=(Y_t)_{t\geq 0}$ and $\bar x\geq Y_0$ we define 
\begin{equation}
	\label{e:overlineY}
	\overline{Y}_t = \bar{x}\vee \sup_{0\leq s\le t}Y_s,
\end{equation}
and call $\overline{Y}=(\overline Y_t)_{t\geq 0}$ the \emph{running supremum} of $Y$ with \emph{initial maximum level} $\bar{x}$. 
We let $(\Omega,\mathcal F)$ be a measurable space and on it we define a family of probability
measures $(\mathbb P_{x,\bar x})_{x\in\mathbb R,\bar x\geq x}$ and a stochastic process
$X=(X_t)_{t\geq 0}$  such that, under $\mathbb P_{x,\bar x}$, $X$ is a {spectrally
negative} Lévy process starting at $x$ and with initial maximum level $\bar x$ (i.e.~$X_0=x$ and $\overline X_0=\bar x$). We refer to Section \ref{sec_levy} below for
some background information on spectrally negative L\'evy processes.
\revised{Let $(\Ff^\circ_t)_{t\geq 0}$ be the natural filtration of $X$, and
$\Ff_t \coloneqq \cap_{s\ge t} \Ff^\circ_t$ its right-continuous enlargement.}
%and define 
%\begin{equation}
%  \label{e:overlineY}
%  \overline{Y}_t = \bar{x}\vee \sup_{0\leq s\le t}Y_s,
%\end{equation}
%The stochastic process $\overline{Y}=(\overline Y_t)_{t\geq 0}$ is referred to
%as the \emph{running supremum} of $Y$ with initial maximum level $\bar{x}$. We
%use the notation 
%\begin{equation*} 
%  \mathbb P_{x,y}(\cdot) 
%  = \mathbb P_x(\cdot \mid \overline X_0=\bar x = y), 
%  \quad \text{for $y\geq x$}.
%\end{equation*} 
The random variable $X_t$ represents the uncontrolled capital level at
time $t$ of an insurance company. We wish to understand the effect of
controlling the capital level by introducing \emph{taxation}, which  yields
revenue for the controller, and \emph{bailouts} (or \emph{capital injections}),
which increase the process but have a cost for the controller. To this end, we
fix a \emph{lower tax rate bound} $\alpha \ge 0$, an \emph{upper tax rate
bound} $\beta \in [\alpha,1)\cap(0,1)$, and a \emph{discount rate} $q > 0$. Next we
define our class of controls.
\begin{definition}\label{def_admis}
  A pair of
  stochastic processes $(H,L)=((H_t)_{t\geq 0},(L_t)_{t\geq 0})$ is called an
  admissible control under $\mathbb P_{x,\bar x}$ if the following holds:
  \begin{enumerate*}
    \item\label{item_admis_adapt}  $H$ and $L$ are left-continuous and adapted to $(\Ff_t)_{t\geq 0}$,
    \item\label{item_admis_bounds} $H_t \in [\alpha,\beta]$ for all $t\ge 0$,
    \item $L$ is increasing (in the weak sense) with  $L_0 = 0$,
    \item\label{item_admis_cont} $\overline{X+L}$ is continuous,
    \item\label{item_admis_expcap} $\EE_{x,\bar{x}} \int_0^\infty e^{-qs} \dd L_s < \infty$,
    \item\label{item_admis_pos}  $\mathbb P_{x,\bar x}(U_{t+}\ge 0$ for all $t\ge 0)=1$ 
      \end{enumerate*}	where $U=(U_t)_{t\geq 0}$ is the controlled process defined by
      \begin{equation}
        \label{e:controlled-V}
        U_t = X_t + L_t - \int_{0^+}^t H_s \, \dd(\overline{X+L})_s.
      \end{equation} 
  We write $\Pi_{x,\bar x}$ for the set of all admissible controls under $\mathbb P_{x,\bar x}$.	
\end{definition}	
For an admissible control $(H,L)$, $H_t$ represents the tax rate at time $t$
and $L_t$ represents the cumulative amount of bailout funds added to $U$ up to
time $t$. From \eqref{e:controlled-V} and Lemma~\ref{l:sup} below, one can see
that taxes are only paid when the controlled capital process $U$ is at its
maximum, i.e., at those times $t$ such that $\overline U_t=U_t$. Regarding the
conditions of an admissible control in Definition \ref{def_admis},
\ref{item_admis_pos} ensures that the controlled capital level cannot be
strictly negative for any duration of time, which reflects that the
controller/government is compelled to bail the company out. We highlight that,
as, the bailout control $L$ is assumed to be left-continuous (rather than
right-continuous), the controlled process $U$ can be strictly negative at a
discrete set of time points.  On the one hand, left-continuous controls are
natural, as the controller can only react after an event, in particular a jump.
On the other, this framework also extends easily to a version of the control
problem in which bailouts are optional instead of mandatory; see, for example,
\cite[Definition 1]{gajek_kucinski}. Conditions \ref{item_admis_cont} and
\ref{item_admis_expcap} in Definition \ref{def_admis} are present for technical
reasons, though they also make sense from a practical point of view:
\ref{item_admis_cont} avoids controls where tax has to be paid over a lump-sum
bailout and  \ref{item_admis_expcap} excludes controls where the expected
discounted value of the total bailouts is infinity, which is a common
assumption in these type of problems; we point once again to \cite[Definition
1]{gajek_kucinski}.

In order to state our optimality criterion, we fix a \emph{bailout penalty
factor} $\eta \ge 0$, an initial maximal capital $\bar x\geq 0$ and an initial
capital $x\leq \bar x$ and we define, for $(H,L) = \pi \in \Pi_{x,\bar x}$, the
\emph{value function}
\begin{equation*}
  v^\pi(x,\bar{x})
  =
  \EE_{x,\bar{x}}\Biggl[ \int_{0^+}^\infty e^{-qs} H_s\, \dd(\overline{X+L})_s
  - \eta \int_0^\infty e^{-qs} \dd L^+_s \Biggr],
\end{equation*}
where $L^+=(L^+_t)_{t\geq 0}$ is the right-continuous version of $L$, i.e.
$L^+_t=L_{t+}$.  
\revised{In this definition, the first term represents the present value of taxation and the second term the present value of bailouts, with the factor $\eta$ representing the a multiplicative additional cost of each bailout.}

We wish to solve the \emph{optimal control problem}
\begin{equation}
  \label{e:control}
  v^*(x,\bar{x})
  =
  \sup_{\pi\in \Pi_{x,\bar x}}
  v^\pi(x,\bar{x}),
\end{equation}
by finding $v^*$ and for each pair $(x,\bar x)$ a choice of $\pi$ which attains it.  

Our contributions are the following.  First, in Theorem~\ref{t:control}, we
solve \eqref{e:control} under a
minor condition on the L\'evy measure of $X$ (equivalent to $\Pi_{x,\bar x}$
being nonempty for some (or equivalently all) $(x,\bar x)$) in the case where $\eta\geq 1$, and show that an optimal control
is given by a threshold tax rate, which corresponds to
$H_t=\alpha+(\beta-\alpha)\mathbf 1_{\{\overline U_t> b\}}$ for some $b\ge 0$,
in combination with minimal bailouts.
%%%%
\begin{figure}%[t]
  \begin{center}
    \includegraphics[scale=0.6]{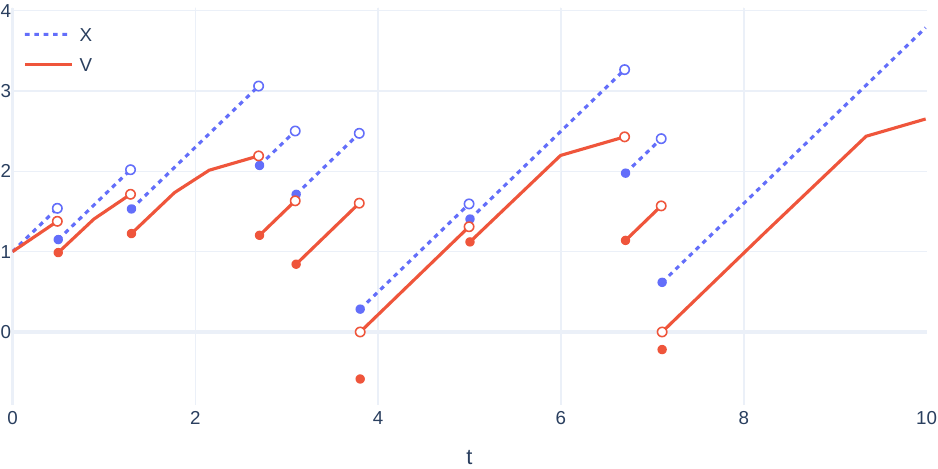}
    \caption{%
      Illustration of the natural tax process $V$ with threshold tax rate at $b$
      and minimal bailouts, where
      $\alpha = 0.3$, $\beta = 0.7$ and $b=2$.
      The blue (dashed) line is the path of the background Lévy process $X$.
      The red line is the process $V$. Note that at times when
      $V$ experiences a bailout, its instantaneous negative
      value is retained and the process is sent to the value zero
      immediately after.
      \label{f:reflection}%
    }
  \end{center}
\end{figure}
%%%%
A sample path of the controlled capital process associated
with such threshold tax rate, denoted by $V$, is provided in Figure \ref{f:reflection}.
The assumption $\eta\geq 1$ makes sense because in practice there is a cost of
capital. But in fact,  Theorem \ref{t:control} would no longer hold when we allow
$\eta<1$: an optimal control for \eqref{e:control} in that case might look very
different, with minimal bailouts no longer being optimal.

Second, we observe that the optimally controlled capital process belongs to the class of
natural tax processes with minimal bailouts with `natural' meaning that the tax rate is a function of (the running supremum of) the
controlled capital process (see also \cite{ALW-tax-equiv}). 
We give a rigorous definition of this class involving a construction and characterisation
of its elements.
If the
initial maximum capital level $\bar x$ is strictly positive, then  this is not hard to do
(see constructions in \cite{IvanovsI,wang2019optimalLastOne}
for specific natural tax rates), because the tax payments and bailouts
occur at distinct times,
and the existence and
uniqueness of the process without bailouts has been shown in
\cite{ALW-tax-equiv}.  However, if $\bar x=0$ and the 
L\'evy process is of unbounded variation, then tax payments and bailouts take
place simultaneously within any time interval $(0,\epsilon)$ with $\epsilon>0$,
and it is less clear how to construct the process.
We do this by introducing a Skorokhod-type problem for
c\`adl\`ag paths and use a contraction argument to prove the
existence and uniqueness of its solution, which we call the
tax-reflection transform. This extends known results for the special case where $X$ is a Brownian motion; see \cite{legall_yor,davis96,davis_newyork,davis99,chaumount_doney} and Remark \ref{remark_perturbed} below. 

Third, whereas in
\cite{LevyInsuranceRiskProcessWithTax,wang2019optimalLastOne}
(respectively \cite{OptimalLossCarryForwardWang} and \cite{wang_zhang}), the authors optimise over tax rate controls
$H$ of the form $H_t=\beta\mathbf 1_{\{\overline U_t> b\}}$ where $b\geq 0$
(respectively $H_t=f(\overline X_t)$ where $f$ is an $[\alpha,\beta]$-valued
measurable function), we allow $H$ to be a more general adapted process. Given
that a natural tax process with (or without) minimum bailouts is not Markov on
its own but only when considered together with its running supremum (see
Proposition \ref{prop_markov} below) this requires one to consider the optimal
value function $v^*$ in \eqref{e:control} as a function of $(x,\bar x)$ and not
just $x$. This makes it more challenging to establish a verification lemma for
solving the optimal control problem in comparison to, e.g., optimal dividend
problems (with or without capital injections) for which the optimally
controlled process is on its own Markovian. Our third contribution is that
we provide a rigorous framework for dealing with optimal taxation problems for
this general class of tax rate controls, which, in principle, should be easily
adaptable to handle other similar optimal taxation problems as well.
Indeed, the first author has shown in \cite{AlGhanim-thesis} that the question
of optimal taxation without bailouts can be handled using this approach.

The last contribution is that we provide an expression, in terms of the scale
functions of the L\'evy process, for the value function
associated with the threshold tax rate and minimal bailouts using the novel
approach developed in the first author's PhD thesis \cite{AlGhanim-thesis}.
This involves a characterisation
lemma which is similar in nature to the verification lemma that we use for the
optimal control problem. When $x=\bar x>0$ such an expression for the value
function has been derived in \cite[Equation (14)]{wang2019optimalLastOne} in
the case where $\alpha=0$ and can be derived from \cite[Theorem 2]{IvanovsII}
by taking $\theta\downarrow 0$ there in the case where the tax rate is constant
(i.e. $\alpha=\beta$); both papers use different methods than ours.

\medskip\noindent
The present article is an elaboration on the work done in the first author's
PhD thesis \cite{AlGhanim-thesis}, where a slightly weaker version of the main
theorem was proved (in which taxation could not start from zero capital).
During preparation of the present work, \cite{wang_zhang} was made available, which
studies a version of the control problem
where it is assumed that the tax rate controls are latent tax rates, bailouts
are assumed to be minimal (so there is no optimisation with respect to the bailout
strategy) and $x=\bar x>0$ in \eqref{e:control}. The optimal control problem in
\cite{wang_zhang} can be seen as a deterministic optimal control problem,
because under these restrictions on the controls,
one can obtain an analytic expression for the value function; this is done
in \cite{wang_zhang} by
making the connection with draw-down reflected L\'evy processes, which is an
alternative to our approach using a characterisation lemma.
%In the case where $\alpha=0$, and the starting point is at the threshold level $b>0$, our expression for the value function in \eqref{e:vtaxinj} agrees with equation (5.3) that derived in \cite{wang_zhang}. As mentioned above, the optimal control problem studied in this paper is more general than the one considered in \cite{wang_zhang}, and our optimal value function is a function of two variables, therefore, we have to use the Meyer-It\^o formula through the verification lemma for verifying optimality. On the other hand, in \cite{wang_zhang}, a martingale is constructed without using It\^o formula in order to prove the optimality for a solution of the HJB equation.   \rlIL{Discuss here differences between two works.}

The rest of the paper is structured as follows. 
\revised{In \cref{s:bailout}, after a few preliminaries on spectrally negative L\'evy processes, we
give an informal definition of the natural tax process with minimal bailouts in
a general setting, and provide expressions for the net present value of tax and
injections associated with a threshold tax rate. In \cref{sec_mainresult}, we state our main result about the solution of the control problem, whereas Section \ref{sec_example} covers an example which is computationally feasible and plots the optimal tax threshold
and value function against various parameters. Most of the proofs are postponed to Section \ref{sec_proofs}. Section \ref{sec:construction} contains the rigorous definition and construction of the natural tax process with minimal bailouts. For this we need to introduce the tax-reflection transform of a c\`agl\`ad path, which is done  in Section \ref{sec:tax-reflection}.
Section \ref{sec_verif_char} consists of the verification lemma used for solving the control problem and the characterisation lemma used to prove, as well as the proof itself of, the expression for the value function associated with a natural tax process with a threshold tax rate and minimal bailouts. Finally, in Section \ref{sec_proof_mainthm} the main result is proved.}

\section{Natural tax processes with minimal bailouts}
\label{s:bailout}

\revised{After some preliminary discussion on Lévy processes, we move on to
describe the natural tax process with minimal bailouts,
and compute a functional representing the net present value of taxation
less bailout cost.}

\subsection{Preliminaries on spectrally negative Lévy processes}\label{sec_levy}

A stochastic process is called a
\emph{Lévy process} if it has stationary,
independent increments and càdlàg paths.
If additionally the paths have no positive jumps,
and are not decreasing (in the weak sense), then $X$ is called a
\emph{spectrally negative Lévy process}.
Recall that $X$ under $\mathbb P_x$ is assumed to be such a process starting at $x$. The \emph{Laplace exponent}
$\psi \from [0,\infty)\to \RR$ defined by
$\EE_x[e^{s X_t}] = e^{s x+ t\psi(s)}$,   has the following \emph{Lévy--Khintchine representation}:
\[
  \psi(s) = cs + \frac{1}{2}\sigma^2s^2
  + \int_0^\infty (e^{-s \theta}-1+s \theta \Indic{\theta\leq 1})\, \nu(\dd \theta),
  \qquad s\ge 0.
\]
Here, $c \in \RR$ incorporates any deterministic drift,
$\sigma \ge 0$ is the volatility of the Brownian component,
and $\nu$ is a measure on $(0,\infty)$ satisfying 
$\int_0^\infty \min\{1, \theta^2\} \, \nu(\dd \theta) < \infty$, which represents
the intensity and magnitude of the jumps. We call $(c,\sigma,\nu)$ the \emph{L\'evy triplet} of $X$.
%
%In the context of insurance, this can be made tangible by considering
%the \emph{Cramér--Lundberg process}.
%Suppose that premia are collected at a constant rate
%$\mathtt{d} > 0$, that claims arrive at rate $\lambda>0$ and have
%magnitude described by a random variable with distribution function $F$.
%Then, the stochastic process describing the company's wealth is
%a spectrally negative Lévy process
%with Laplace exponent
%% If $N$ is a Poisson process with intensity $\mu$;
%% $(\xi_n)_{n\ge 1}$ is a collection of independent,
%% identically divisible random variables with distribution
%% function $F$ having support $[0,\infty)$;
%% $\mathtt{d},\sigma\ge 0$; and $B$ is a standard Brownian motion,
%% then
%% \[ 
%%   X_t = \mathtt{d} t + \sigma B_t - \sum_{n=1}^{N_t} \xi_n, 
%%   \qquad t\ge 0, 
%% \]
%% is a snLp, with Laplace exponent
%\[
%  \psi(s)
%  =
%  \mathtt{d} s
%  %   + \frac{1}{2}\sigma^2s^2
%  + \lambda \int_0^\infty (e^{-s x}-1) \, F(\dd x),
%  \qquad s \ge 0,
%\]
%In general, spectrally negative Lévy processes can be constructed by taking a limit of processes
%of the above form as more intense jumps of small size are included;
%this construction is known as the Lévy--Itô decomposition \cite[Chapter 2]{Kyp2}.
%
If the L\'evy triplet is such that $\sigma=0$ and $\int_0^1\theta\nu(\mathrm d \theta)<\infty$ then $X$ has paths of bounded variation; otherwise the paths of $X$ have unbounded variation.

%We introduce the generator $\Aa$ of $X$, an operator with action
%\begin{equation}
%  \label{e:generator}
%  \Aa f(x)
%  =
%  \frac{1}{2}\sigma^2 f''(x)
%  + c f'(x)
%  + \int_0^\infty (f(x-y) - f(x) + f'(x)\theta\Indic{\theta\leq 1} ) \, \nu(\dd \theta).
%\end{equation}
%By a slight abuse of notation, we also use $\Aa$ to denote
%the operator acting on functions of two variables,
%by defining $\Aa f(x,y) = \Aa[f(\cdot , y)](x)$.

Along with the Laplace exponent, another quantity which appears occasionally
is its right-inverse
\[
  \Phi(q) = \sup\{ s \ge 0: \psi(s) = q\}.
\]
Associated with the spectrally negative Lévy process $X$ are its so-called
$q$-scale functions, given for any $q \ge 0$.
\revised{We describe these in some detail here because they form part of the expression
of the optimal value function.}
The first is a right-continuous
function $W^{(q)} \from \RR \to [0,\infty)$,
whose Laplace transform is given by
\[
  \int_0^\infty e^{-s x} W^{(q)}(x)\, \dd x
  = \frac{1}{\psi(s) - q}, \quad s>\Phi(q),
\]
%for $s$ sufficiently large,
setting $W^{(q)}(x) = 0$ for $x< 0$. With 
\begin{equation*}
  \overline W^{(q)}(x)=\int_0^x W^{(q)}(y)\mathrm d y, \quad x\geq 0,
\end{equation*}
the second is defined
as
\[
  Z^{(q)}(x) = 
  \begin{cases}
    1 + q \overline W^{(q)}(x) , & x \ge 0, \\
    1, & x < 0.
  \end{cases}
\]
Finally, it will be notationally convenient to also define
an antiderivative of $Z$,
\[
  \overline{Z}^{(q)}(x) 
  =
  \begin{cases}
    \int_0^x Z^{(q)}(y) \, \dd y, & x \ge 0, \\
    x, & x < 0.
  \end{cases}
\]
The scale function $W^{(q)}$ is a strictly increasing on $[0,\infty)$ and if
$X$ has paths of bounded variation, then $W^{(q)}$ is continuous on
$\mathbb R\backslash\{0\}$, whereas if $X$ has paths of unbounded variation, then
$W^{(q)}$ is continuous on $\mathbb R$ and continuously differentiable on
$\mathbb R\backslash\{0\}$. 
Further, if $q>0$, then
\begin{equation}\label{limit_ratio_scale}
  \lim_{x\to\infty} \frac{Z^{(q)}(x)}{W^{(q)}(x)} = \frac q{\Phi(q)},
\end{equation}
see \cite[Lemma~3.3]{KKR-scale},
For more background information on spectrally negative L\'evy processes and
their scale functions, \revised{including their numerical computation,}
we refer to \cite[Chapter 8]{Kyp2} and \cite{KKR-scale}.

\subsection{Definition and functionals of natural tax processes with minimal bailouts}
\label{sec:natural}

\revised{
We now give a brief description of the process which will form the controlled process under the optimal strategy, postponing a full discussion with proofs until Sections \ref{sec:tax-reflection} and \ref{sec:construction}.
Suppose that $\delta \colon [0,\infty) \to [0,1)$ is a measurable function,
such that the ordinary differential equation
\begin{equation}
  y'(t) = 1-\delta(y(t)), \quad t\ge 0, \qquad y(0) = y_0,
  %\bar{x}
  \label{e:ode-a}
\end{equation}
has a solution for every $y_0\ge 0$, in the sense that there is an absolutely continuous
function $y$ for which \eqref{e:ode-a} holds for almost every $t$.
The relevance of this differential equation arises from \cite[Theorem 1]{ALW-tax-equiv}.
Equation \eqref{e:ode-a} has unique solutions when $\delta$ is increasing
\cite[Example 2]{ALW-tax-equiv}, and this will be the case for our candidate solutions.
Let $x\in\RR$ and $\bar{x}\ge x \vee 0$.
We work under the measure $\PP_{x,\bar{x}}$ given in the introduction,
under which $X$ is a spectrally negative Lévy process with
$(X_0, \overline{X}_0) = (x,\bar{x})$
Recall also the notation \eqref{e:overlineY} for running maxima,
which depend on the variable $\bar x$.
For such $\delta$, we will show in Section \ref{sec:construction} that there exists
a pair $(V,K)$ of stochastic processes, with the property that, almost surely,
$V$ possesses left and right limits at every time; $K$ is left-continuous;
$(V_0, \overline{V}_0, K_0) = (x,\bar{x},0)$; $V_{t+}\ge 0$ for all $t\ge 0$;
and 
\begin{equation}\label{VK_eq_revised}
\begin{split}
	V_t = & X_t + K_t - \int_{0^+}^t \delta(V_s) \, \dd (\overline{X+K})_s  \quad t\geq 0, \\
	0 = & \int_{0}^\infty \Indic{V_t >0} \, \dd K_{t+} 
\end{split}
\end{equation}
Moreover, $(V,K)$ is adapted to the filtration $(\Ff_t)_{t\geq 0}$, and
the bivariate process $(V,\overline{V})$ is strong Markov in $(\Ff_t)_{t\geq 0}$.
We call the process $V$, or the process $(V,K)$, the \emph{natural tax process
with minimal bailouts}, with the \emph{(natural) tax rate} $\delta$.
The interpretation of \eqref{VK_eq_revised} is that when $V$ is drifting
at its running maximum at level $v$, a proportion $\delta(v)$ is taken from
the (uncontrolled) increase, whereas increases in $K$ correspond to bailouts which  cannot take place when $V_t>0$ and are there to ensure that the condition $V_{t+}\ge 0$ is met.
}

As outlined in the introduction, the solution of the optimal control 
problem is such a tax process, with a specific threshold
tax rate. 
The main result that we need on this type of process is a description of its
value function. We will use this in the next section, varying
the threshold, to obtain the solution of the control problem.
For $\gamma \in [0,1)$ and $x\geq 0$, define 
\begin{equation}\label{def_Rgam}
  R_\gamma(x) = 
  \begin{cases}
    \frac\gamma{1-\gamma} Z^{(q)}(x)^{\frac1{1-\gamma}}\int_x^\infty Z^{(q)}(y)^{-\frac1{1-\gamma}}(1-\eta Z^{(q)}(y))\mathrm d y & \text{if $\gamma\in(0,1)$}, \\
    0 & \text{if $\gamma=0$},
  \end{cases}
\end{equation}
noting that, when $\gamma>0$, the integral above converges because, as $q>0$,
there exists $s>0$ such that $\lim_{x \to\infty} e^{-s x} Z^{(q)}(x) =\infty$,
see \eqref{lowerbound_Zq} below. 

The condition $\int_1^\infty \theta \, \nu(\dd\theta)<\infty$ appearing in the following
result is equivalent to the condition $\EE[X_1] = \psi'(0) > -\infty$, and since the
latter quantity appears in the value function, we need it to be finite.
\revised{In the absence of this condition, the value of bailouts would be infinite.}
\begin{proposition}\label{p:value}
  Assume that $\int_1^\infty \theta\,\nu(\dd \theta) < \infty$.
  Let $(V,K)$ be
  the natural tax process with minimal bailouts with tax rate given by the
  threshold tax rate $\delta_b$ with threshold level $b\geq 0$, which is
  defined by
  \begin{equation*}%\label{def_thresholdtax}
    \delta_b(z) 
    = \alpha\mathbf{1}_{\{z\le b\}} + \beta\mathbf{1}_{\{z>b\}},
    \quad z\geq 0.
  \end{equation*}
  Then for $(x,\bar x)\in\mathbb R\times[0,\infty)$ such that $x\leq \bar x$ and for $\eta\geq 0$,
  \revised{the control $\pi = (\delta_b(V),K)$ is admissible and}
  \begin{align}
    \revised{v^\pi(x,\bar{x})}
    &=
    \EE_{x,\bar{x}}
    \Biggl[
      \int_0^\infty e^{-qs} \delta_b(V_s) \, \dd (\overline{X+K})_s
      -
      \eta
      \int_0^\infty e^{-qs} \, \dd K_s^+
    \Biggr] \nonumber \\
    &= \eta \left( \overline Z^{(q)}(x)+\frac{\psi'(0)}q \right) \nonumber \\
    & \quad {} + \frac{Z^{(q)}(x)}{Z^{(q)}(\bar x)} \left\{ R_\alpha(\bar x) + \left(  \frac{Z^{(q)}(\bar x)}{Z^{(q)}(\bar x\vee b)} \right)^{\frac1{1-\alpha}} ( R_\beta(\bar x \vee b) - R_\alpha(\bar x\vee b) )  \right\}.
    \label{e:vtaxinj}
  \end{align}
\end{proposition}
Note that \eqref{e:vtaxinj} does not hold in the (excluded) case $\alpha=\beta=0$, which corresponds to having no taxation at all. 
%In section~\ref{s:control}, we will prove this proposition as a consequence of a characterisation result (Lemma~\ref{l:char}) for functionals of similar form. 
%%%
Proposition~\ref{p:value} will be proved in Section~\ref{s:containing-p:value}, after we have established Lemma~\ref{l:char}, which gives a characterisation of the value function. We only mention for now that $\delta_b$ satisfies the conditions in Theorem \ref{p:V}, see \cite[Example 2]{ALW-tax-equiv}.

\section{Solution of the optimal control problem}
\label{s:control}

\subsection{Main result}\label{sec_mainresult}

Before stating the main result we need one more piece of notation:
\begin{equation}\label{def_CQ}
  \begin{split}
    C(b) & =  Z^{(q)}(b)^{-\frac1{1-\alpha}} ( R_\beta(b) - R_\alpha(b) ) , \\
    Q(b) & =  Z^{(q)}(b)^{-\frac1{1-\alpha}} \left( \frac{Z^{(q)}(b)}{Z^{(q)\prime}(b)} (1-\eta Z^{(q)}(b))  - R_\alpha(b)   \right), \\
    b^* & =  \inf \left\{ b>0 : R_\beta(b) < \frac{Z^{(q)}(b)}{Z^{(q)\prime}(b)} (1-\eta Z^{(q)}(b)) \right\} \\
    & =  \inf \left\{ b>0 : C(b) < Q(b) \right\} , 
  \end{split}
\end{equation}
where $R_\gamma(\cdot)$ was defined in \eqref{def_Rgam}.
%%%
\revised{Next we present our main result as well as a lemma and a follow-up remark containing some properties of $b^*$. The proof of the  main result is postponed to Section \ref{sec_proof_mainthm}.}
\begin{theorem}\label{t:control}
  Assume that $\int_1^\infty\theta  \, \nu(\dd \theta) < \infty$ and that $\eta\ge 1$.
  Then, $b^*<\infty$ and for each $(x,\bar x)\in\mathbb R\times[0,\infty)$ with $x\leq \bar x$ an optimal control
  for \eqref{e:control} is given by the pair $(\delta_{b^*}(V),K)$ where $(V,K)$ is the natural tax process with minimal bailouts and threshold tax rate $\delta_{b^*}$. The optimal value function $v^*(x,\bar x)$ %for $(x,\bar x)\in\mathbb R\times[0,\infty)$ with $x\leq \bar x$ 
  is expressed by the right hand side of \eqref{e:vtaxinj} with $b=b^*$. 
  %We can further say the following about $b^*$: if $\eta=1$ then $b^*=0$;  if $\eta>1$ and $X$ is of unbounded variation  then $b^*>0$; if $b^*>0$ then it is the unique point where $Q(\cdot)$ and $C(\cdot)$ are equal and, finally, $b^*$ does not depend on $\alpha$.
\end{theorem}

 Using the analogue of \eqref{Rgamderiv} for $R_\gamma$, we can derive the following ODE for $C(\cdot)$,
\begin{equation}\label{ODE_C}
  C'(b) = \left( \frac1{1-\beta} - \frac1{1-\alpha} \right) \frac{Z^{(q)\prime}(b)}{Z^{(q)}(b)} (C(b)-Q(b)), \quad b>0.
\end{equation}
Equation \eqref{e:vtaxinj} suggests that the threshold level that yields the highest value function amongst the threshold tax rate with minimal bailouts strategies is the one that maximises the function $C(\cdot)$. The next lemma shows that $b^*$ is the unique maximiser of this function. 
\begin{lemma}\label{lem_CQ}
  We have (i) $b^*<\infty$, (ii) $C'(b)>0$ and $C(b)>Q(b)$  for $b\in(0,b^*)$ and  (iii) $C'(b)<0$ and $C(b)<Q(b)$ for $b\in(b^*,\infty)$. 
\end{lemma}	
\begin{proof}
  We will first prove a monotonicity property for the function $Q(\cdot)$ using arguments inspired 
  by pp.~15--16 in \cite{APP2008}.
  We will need the following identity which follows from  Theorem~1 in \cite{AKP-exit}:
  \begin{equation}\label{avrametal}
    \mathbb E_{0,0} \left[ e^{-q \hat\tau_a} \right] =  Z^{(q)}(a) - W^{(q)}(a)\frac{q W^{(q)}(a)}{W^{(q)\prime}(a)}
    =  Z^{(q)}(a) -  \frac{Z^{(q)\prime}(a)^2}{Z^{(q)\prime\prime}(a)},
  \end{equation}
  where $\hat\tau_a \coloneqq \inf\{t\geq 0: \overline X_t - X_t>a\}$, 
  $Z^{(q)\prime\prime}(y)=q W^{(q)\prime}(y)$
  and $W^{(q)\prime}(y)$ denotes the right-derivative of $W^{(q)}(y)$.
  Assume initially that $\alpha>0$.
  We have by an integration by parts and \eqref{avrametal},
  \begin{equation*}
    \begin{split}
      Q(b) = & \int_b^\infty  Z^{(q)}(y)^{-\frac\alpha{1-\alpha}}   \left( \eta  + \frac{Z^{(q)\prime\prime}(y)}{Z^{(q)\prime}(y)^2}(1-\eta Z^{(q)}(y)) \right) \mathrm d y \\
      = & -\eta \int_b^\infty  Z^{(q)}(y)^{-\frac\alpha{1-\alpha}} \frac{Z^{(q)\prime\prime}(y)}{Z^{(q)\prime}(y)^2}   \left( \mathbb E_{0,0} \left[ e^{-q \hat\tau_y} \right] - \frac1\eta \right) \mathrm d y,
    \end{split}
  \end{equation*}	
  Since $\mathbb E_{0,0} \left[ e^{-q \hat\tau_a} \right]$ is decreasing in $a$,
  % has the property that it is constant on $[0,\varepsilon]$ and strictly decreasing on $(\varepsilon,\infty)$ for some $\varepsilon\in[0,\infty)$ \fbox{TODO}, it follows that with $a^*:=\inf\left\{a\geq 0: \mathbb E_{0,0} \left[ e^{-q \hat\tau_a} \right] < \frac1\eta \right\}$, $Q(\cdot)$ is (i) either strictly increasing or constant on $(0,a^*)$ and (ii) strictly decreasing on $(a^*,\infty)$.
  it follows that there exist $0\leq a_1\leq a_2<\infty$ such that $Q(\cdot)$ is strictly increasing on $(0,a_1)$, constant on $(a_1,a_2)$ and strictly decreasing on $(a_2,\infty)$. 
  If on the other hand $\alpha=0$, then we have for a fixed $a\in(0,\infty)$,
  \begin{equation*}
    Q(b) = Q(a) -\eta \int_b^a  \frac{Z^{(q)\prime\prime}(y)}{Z^{(q)\prime}(y)^2}   \left( \mathbb E_{0,0} \left[ e^{-q \hat\tau_y} \right] - \frac1\eta \right) \mathrm d y
  \end{equation*}
  and so we can use the same argument as in the $\alpha>0$ case to deduce the same monotonicity property for $Q(\cdot)$.

  Note that by \eqref{ODE_C}, $C'(b)>0$ if and only if $C(b)>Q(b)$ and  $C'(b)<0$ if and only if $C(b)<Q(b)$.
   Using \eqref{limit_ratio_scale} together with
  l'H\^opital's rule in \eqref{def_CQ} yields
  \[ 
    \lim_{b\to\infty}C(b)=\lim_{b\to\infty}Q(b) 
    = 
    \begin{cases} 
      0 & \text{if $\alpha>0$} \\ 
      -\eta/\Phi(q) & \text{if $\alpha=0$.}
    \end{cases}
  \]
  Because $Q(\cdot)$ is strictly decreasing on $(a_2,\infty)$ it follows that $C(b)<Q(b)$ for $b>a_2$ as otherwise by \eqref{ODE_C}, $C(\cdot)$ will be from some point onwards increasing and thus larger than $Q(\cdot)$ and then $C(\cdot)$ and $Q(\cdot)$ do not converge to the same limit at infinity. 
  Consequently,  $b^*\leq a_2<\infty$ and $C(b)<Q(b)$ for all $b>b^*$ since, assuming without loss of generality $b^*<a_2$, $Q(b)\geq Q(b^*)$ for $b\in(b^*,a_2)$ and $C'(b)<0$ if $C(b)<Q(b)$.  It remains to show that $C(b)> Q(b)$ for $b\in(0,b^*)$ which we prove by contradiction. Suppose this does not hold. Since by definition of $b^*$, $C(b)\geq Q(b)$ for $b\in(0,b^*)$, we must have then that there exists $0<b_1<b_2\leq b^*$ such that $C(b)=Q(b)$ for all $b\in[b_1,b_2]$. Consequently, by \eqref{ODE_C}, $C'(b)=0$ and thus $Q'(b)=0$ for all $b\in(b_1,b_2)$. Therefore, $Q(\cdot)$ is decreasing on $(b_1,\infty)$ by the monotonicity property of $Q(\cdot)$.
  So by \eqref{ODE_C} and recalling $0\leq\alpha<\beta<1$, we have for $b>b_1$,
  \begin{equation*}
    \begin{split}
      C'(b) &= C'(b)+ \left( \frac1{1-\beta} - \frac1{1-\alpha} \right) \frac{Z^{(q)\prime}(b)}{Z^{(q)}(b)}(Q(b_1) - C(b_1)) \\
      & = \left( \frac1{1-\beta} - \frac1{1-\alpha} \right) \frac{Z^{(q)\prime}(b)}{Z^{(q)}(b)} (C(b)-C(b_1)-(Q(b)-Q(b_1)) \\
      & \geq \left( \frac1{1-\beta} - \frac1{1-\alpha} \right) \frac{Z^{(q)\prime}(b)}{Z^{(q)}(b)} (C(b)-C(b_1)).
    \end{split}
  \end{equation*}
  Hence by Gr\"onwall's inequality, $C(b)-C(b_1)\geq 0$ for all $b>b_1$ and thus $C(b)\geq C(b_1)=Q(b_1)\geq Q(b)$ for all $b>b_1$ which forms a contradiction with $b^*<\infty$.
\end{proof}
\begin{remark}\label{remark_propb*}
  From \eqref{def_CQ}, Lemma \ref{lem_CQ} and its proof we can easily deduce
  the following properties of $b^*$. 
  \begin{enumerate}
    \item
      If $\eta=1$, then $Q(b)$ is a strictly decreasing function on
      $(0,\infty)$ and then so must $C(b)$ be, which implies $b^*=0$. 
    \item
      If $\eta>1$ and $X$ has paths of unbounded variation, then $W^{(q)}(0)=0$
      and so
      $\lim_{b\downarrow 0}\frac{Z^{(q)}(b)}{Z^{(q)\prime}(b)} (1-\eta Z^{(q)}(b))=-\infty$,
      whereas $R_\beta(0)\in\mathbb R$, which implies $b^*>0$. 
    \item
      Further, if $b^*>0$ then it must be the unique point $b\in(0,\infty)$
      such that $C(b)=Q(b)$ or equivalently
      $R_\beta(b)=\frac{Z^{(q)}(b)}{Z^{(q)\prime}(b)} (1-\eta Z^{(q)}(b))$. We
      also remark that from the definition of $b^*$ we see that it does not
      depend on the lower tax rate bound $\alpha$.
      % Hmm, not so obvious, will check:
%     \item
%       \revised{Since $R_\beta$ is decreasing in $\eta$, it follows that $b^*$ is increasing
%       in $\eta$, which is intuitive: when bailouts cost more, raising the threshold
%     helps to avoid these costly events.}
  \end{enumerate}
\end{remark}

\subsection{Example}\label{sec_example}

In this section, we consider a specific example in which $X$ is a Lévy process
with Laplace exponent $\psi$ given by
\[
  \psi(s) = (1+\theta)\frac{2\lambda}{\mu} s
  + \frac{1}{2}\sigma^2 s^2
  + \lambda \int_0^\infty (e^{-sx} -1) \mu x e^{-\mu x} \, \dd x,
  \qquad s\ge 0,
\]
meaning that the process experiences negative jumps whose magnitude has a gamma
(or Erlang) distribution with shape parameter $2$ and rate $\mu>0$.  Jumps
arrive at rate $\lambda>0$, and the process also experiences Gaussian
fluctuations with volatility $\sigma>0$.  The parameter $\theta\in \RR$ is the
loading factor, and $\EE_0 X_1 = \psi'(0) = 2\theta\lambda/\mu$.
Unless otherwise mentioned, we will take $\lambda=1$, $\mu=2$, $\theta=0.1$ and
$\sigma=1$ for the parameters associated with $X$, and $\alpha=0.3$,
$\beta=0.6$ and $\eta = 1.25$ for those associated with taxation.

Our goal is to investigate numerically the effect of changing the parameters
$\alpha$, $\beta$, $\eta$ and $q$ on the optimal threshold $b^*$ and the value
function $v$.  Computation of both of these requires evaluating not just the
scale functions of the models in question, but also the function $R_\gamma$
defined in \eqref{def_Rgam}.
This function contains an integral over an unbounded domain, and established
quadrature routines may struggle with it. Therefore, we first show that the
integrand in question decays fast and hence show how to approximate it with a
finite integral.

Taking \cite[Corollary 8.9]{Kyp2} and integrating yields the identity
\[
  \overline{W}^{(q)}(x)
  = \frac{\Phi'(q)}{\Phi(q)} (e^{\Phi(q)x}-1) 
  - \int_0^\infty e^{-qt} \PP(X_t \in [-x,0]) \, \dd t,
  \qquad x\ge 0.
\]
and bounding the probability by $0$ or $1$ gives a simple
bound for $\overline{W}^{(q)}(x)$ in terms of $\Phi$
and its derivative. In particular, it implies that
\begin{equation}\label{lowerbound_Zq}
  Z^{(q)}(x) \ge q\frac{\Phi'(q)}{\Phi(q)}(e^{\Phi(q)x}-1)
\ge q\frac{\Phi'(q)}{\Phi(q)}e^{(\Phi(q)-\delta)x},
\qquad
x \ge \delta^{-1}\log 2,
\end{equation}
for any $0<\delta<\Phi(q)$.
Applying this and similar arguments, we can obtain a bound for the
integrand in the definition of $R_\gamma$ for $\gamma\in(0,1)$:
\begin{multline}
  0\le
  Z^{(q)}(x)^{-\frac{1}{1-\gamma}}
  (\eta Z^{(q)}(y)-1)
  \le
  (\eta-1)\biggl[
    \frac{\Phi(q)}{q\Phi'(q)} e^{-(\Phi(q)-\delta)x}
  \biggr]^{\frac{1}{1-\gamma}}
  \\
  {} + 
  \eta\biggl[
    \frac{\Phi(q)}{q\Phi'(q)} e^{-(\Phi(q)-\delta)x}
  \biggr]^{\frac{\gamma}{1-\gamma}},
  \qquad x\ge \delta^{-1}\log 2.
  \label{e:R-bound}
\end{multline}
Therefore, for example, if we set $\epsilon>0$ and select
\[
  M > \frac{1-\gamma}{\Phi(q)-\delta}
  \biggl[
    \log(1-\gamma) - \log(\Phi(q)-\delta)
    + \frac{1}{1-\gamma} \log\biggl(\frac{\Phi(q)}{q\Phi'(q)}\biggr)
    +\log(\eta-1)-\log \epsilon
  \biggr],
\]
then
\[ \int_M^\infty 
  (\eta-1)\biggl[
    \frac{\Phi(q)}{q\Phi'(q)} e^{-(\Phi(q)-\delta)x}
  \biggr]^{\frac{1}{1-\gamma}}
  \, \dd x < \epsilon,
\]
provided that also $M>\delta^{-1} \log 2$.  Similar considerations allow us to
bound the integral of the second term in \eqref{e:R-bound}, as well as a
similar integral appearing in the value function.  In what follows, we have
chosen $\epsilon = 10^{-10}$ and used this approximation. On top of this, the
quadrature routine that we used yielded an absolute error estimate of at most
$1.07\times 10^{-6}$.  

Previous work on computing the value function in related models appears in
\cite[Section~3]{LundbergRiskProcessWithTax} and \cite[p.~8]{IvanovsII}, the
risk model in both cases being a compound Poisson process with negative jumps
having an exponential distribution. In the former case, this leads to an
explicit expression for the value function in terms of hypergeometric and
elementary functions; however, such a nice formula does not appear to be
possible with Erlang jump distribution.

Returning to our specific example, it is readily seen that $\psi$ is a rational
function of $s$, which implies that the scale functions can be found by partial
fraction decomposition and consist of mixtures of exponentials. However, the
exact decomposition is quite complicated to write down, with both the rates and
coefficients involving roots of polynomials, and for computation, we used
Mathematica to compute the inverse Laplace transform required before importing
the result into sage.  We took a similar approach to the computation of
$\Phi(q)$, required for our analysis above, which also involves roots of a
polynomial depending on the characteristics of the process.

\begin{figure}%[p]
  \centering
  \begin{subfigure}[b]{0.4\textwidth}
    \includegraphics[width=\textwidth]{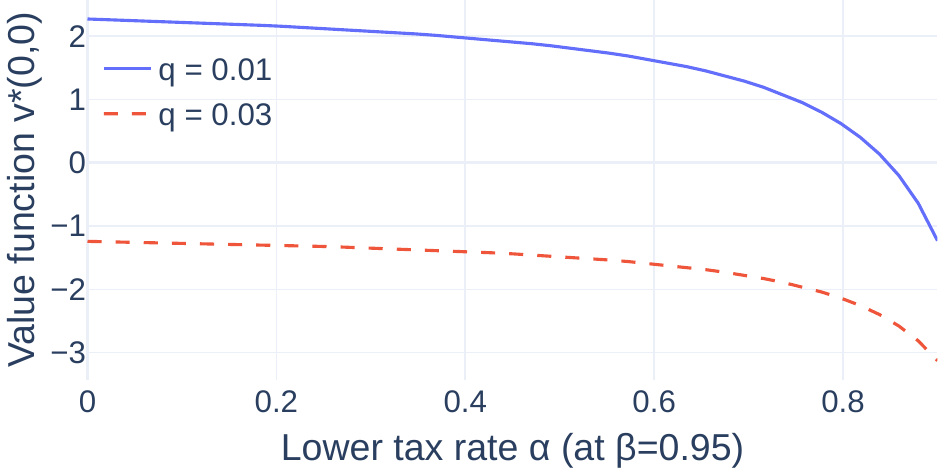}
    \label{f:alpha}
  \end{subfigure}\hfill
  \begin{subfigure}[b]{0.4\textwidth}
    \includegraphics[width=\textwidth]{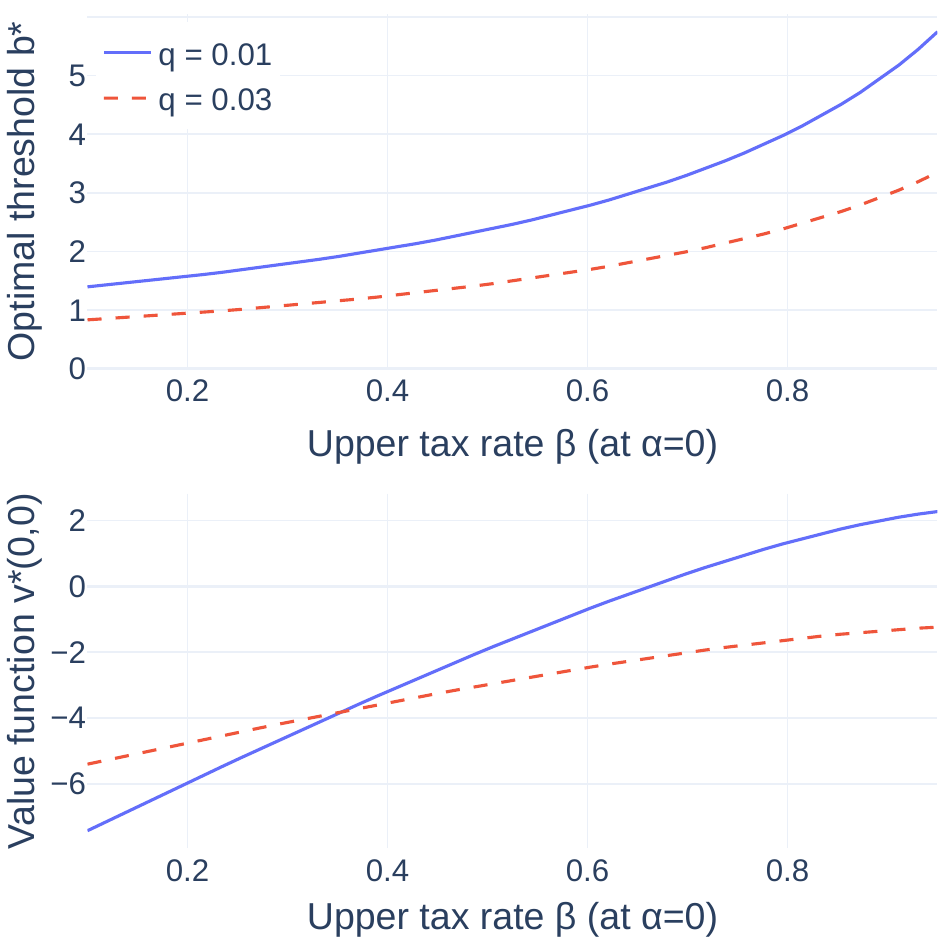}
    \label{f:beta}
  \end{subfigure}
  \caption{Effect of tax rate bounds on the optimal threshold $b^*$
    and the optimal value function at zero $v^*(0,0)$.
  The value of $\alpha$ has no effect on $b^*$.}
  \label{f:tax}
\end{figure}

\begin{figure}%[p]
  \centering
  \begin{subfigure}[b]{0.4\textwidth}
    \includegraphics[width=\textwidth]{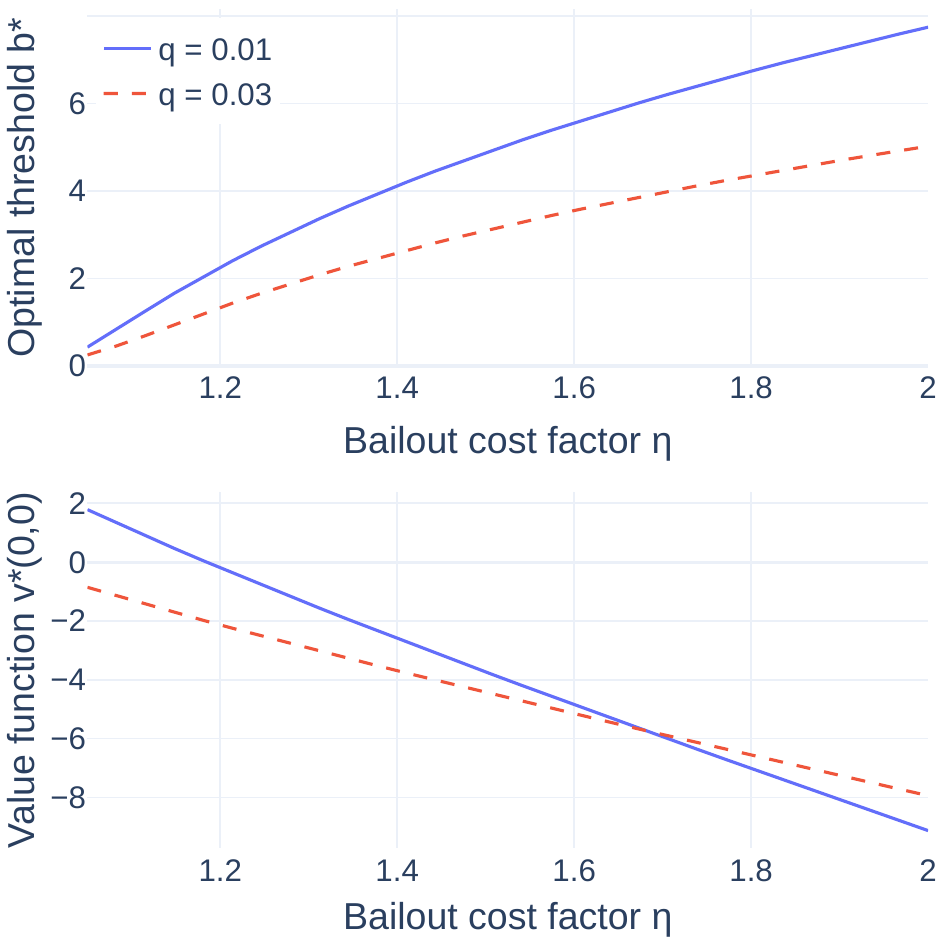}
    \label{f:eta}
  \end{subfigure}\hfill
  \begin{subfigure}[b]{0.4\textwidth}
    \includegraphics[width=\textwidth]{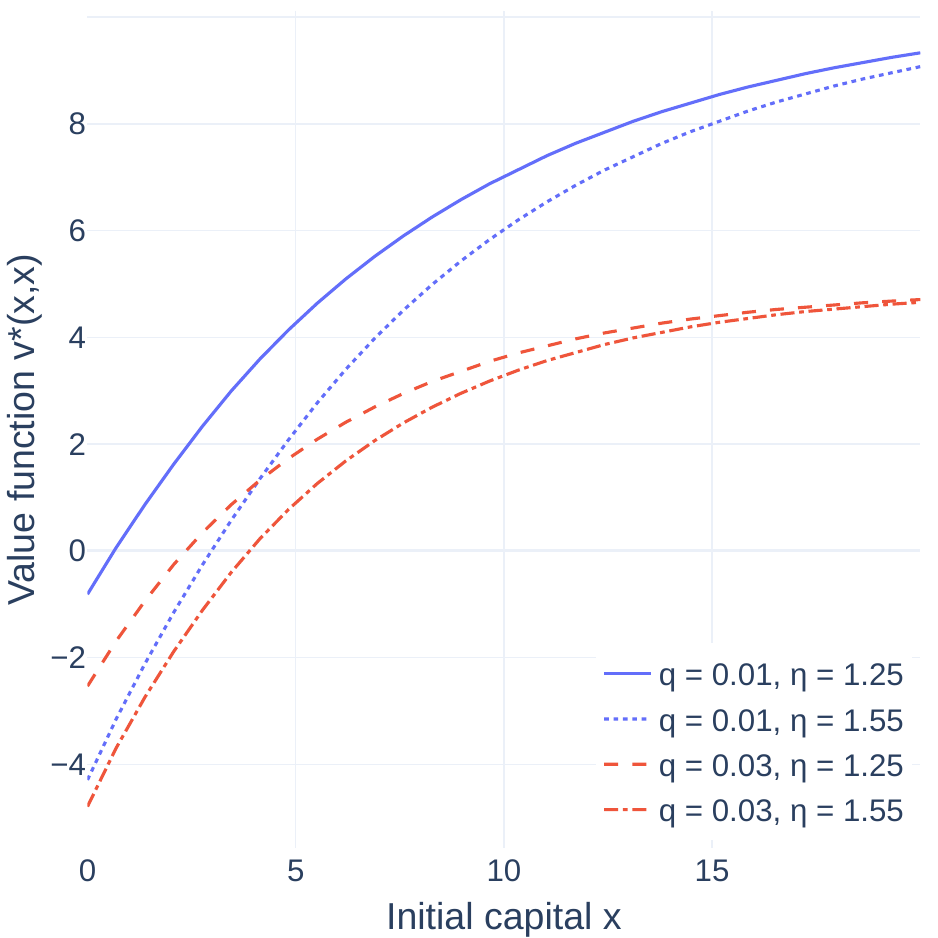}
    \label{f:v}
  \end{subfigure}
  \caption{Effect of varying $\eta$ on the optimal threshold
    $b^*$
    and the optimal value function at zero $v^*(0,0)$, and on
  the value function $v^*(x,x)$.}
  \label{f:eta-v}
\end{figure}

Figures~\ref{f:tax} and~\ref{f:eta-v} show the effect of varying the tax rates
$\alpha$ and $\beta$, and the effect of varying the bailout cost factor $\eta$,
respectively. Where the optimal value function $v^*(0,0)$ is shown, numerically
$10^{-8}$ is used for zero, since our implementations of $v^*(x,x)$ are for the
case $x>0$.
It is notable that both the threshold and the optimal value function at zero
appear to be monotone in $\eta$, and that the optimal value function is not
monotone in $q$, in that the graphs of $v^*(0,0)$ for different values of $q$
intersect. We also remark that the optimal value function $v^*(x,x)$ appears to
be concave, as may be expected from past work on the dividend model with
capital injections \cite[Lemma~3]{APP2008}.

\section{Proofs}\label{sec_proofs}

\subsection{The tax-reflection transform}
\label{sec:tax-reflection}

In this section, we make no distributional assumptions about an underlying risk
process, and instead explain rigorously how to introduce taxation and bailouts
to any càdlàg path $X = (X_t)_{t\ge 0}$. This needs particular care when the
path starts from level zero because then it is possible that taxation and
bailouts are applied simultaneously.
\revised{As we outlined earlier, a key element here is the differential equation
\begin{equation}
  y'(t) = 1-\delta(y(t)), \quad t\ge 0, \qquad y(0) = y_0,
  %\bar{x}
  \label{e:ode}
\end{equation}
where $\delta\colon [0,\infty) \to [0,1)$ is a measurable function.
We call $y$
a solution if it is absolutely continuous and satisfies
\eqref{e:ode} for almost every $t$.
Recall once again the notation \eqref{e:overlineY} in which the running
maximum of paths depends on the variable $\bar x$.}

\begin{theorem}
  \label{p:V}
  Let $\delta\colon [0,\infty) \to [0,1)$ be a measurable function
  such that \eqref{e:ode} has a unique solution for every $y_0\ge 0$.
  Assume, in addition, that
  there exist some $\epsilon>0$ and $\gamma\in[0,1)$ such that 
  $\delta(x)=\gamma$ for $x\in[0,\epsilon]$. 
  %Fix $x,\bar{x} \in \RR$ such that $x\vee 0\le \bar{x}$.
  For every càdlàg path $X$ with no positive jumps and for every 
  $\bar{x}\geq X_0\vee 0$, 
  there exists a unique pair of càdlàg paths $(V^+,K^+)$ with $V^+$ 
  positive and $K^+$ increasing such that 
  %$K^+_{0} = 0$,
  \begin{equation}\label{taxcondition}
    V^+_t = \ X_t + K^+_t - \int_{0^+}^t \delta(V^+_s)  \,\dd (\overline{X+K^+})_s
    , \quad t \ge 0,
  \end{equation}
  % where $(\overline{X+K^+})_t = \bar{x} \vee \sup_{s\le t} (X_s+K^+_s)$,
  and the following complementarity condition is satisfied:
  \begin{equation}\label{compl_condition}
    \int_{0}^\infty \Indic{V^+_t >0} \, \dd K^+_t =   \ 0,
  \end{equation}
  We call $(V^+,K^+)$ the tax-reflection transform of the path $X$ with tax rate $\delta$
  and initial maximum level $\bar{x}$.
  %Moreover, $(V^\infty,K)$ is the only pair of paths satisfying these conditions for a given path of $X$.
\end{theorem}
\begin{proof}
  Throughout the proof we use the notation $\mathcal S(Y)=(\mathcal
  S_t(Y))_{t\geq 0}$ where $\mathcal S_t(Y)=\sup_{0\leq s\leq t}Y_s$ for a
  c\`adl\`ag path $Y$ in order to distinguish it from the notation $\overline Y_t$
  in \eqref{e:overlineY}.  We assume $\bar{x}=X_0\vee 0$ for now leaving the
  case $\bar{x}>X_0\vee 0$ until the end.  We treat the cases $X_0>0$ and
  $X_0\leq 0$ separately. 

  First assume $\bar{x} =X_0>0$. Then the unique pair of paths $(V^+,K^+)$
  satisfying the desired properties can be constructed recursively in a similar
  way as in the Appendix of \cite{IvanovsI}. To explain, given a c\`adl\`ag path
  $Y$ with no positive jumps we define a number of transformations of $Y$.
  First, by Theorem~1(ii) and Lemma~4 in \cite{ALW-tax-equiv} and the first
  assumption in the theorem there exists a unique solution
  $\Upsilon=\Upsilon(Y)$ to the equation
  \begin{equation*}%\label{naturaltaxprocess}
    \Upsilon_t = Y_t - \int_{0^+}^t \delta(\Upsilon_s)\mathrm d \mathcal S_s(Y) , \quad t\geq 0.
  \end{equation*}
  %  where $\overline{Y}_t = \bar{x}\vee \sup_{s\le t} Y_s$.
  Second, $(\Psi,\Phi)=(\Psi,\Phi)(Y)$ is the reflection map defined by
  \begin{equation*}
    \begin{split}
      \Psi_t = & \left( -\inf_{0\leq s\leq t}Y_s \right) \vee 0 , \\
      \Phi_t = & Y_t + \Psi_t .
    \end{split}
  \end{equation*}
  Note that $\Phi_t(Y)$ is the path $Y$ reflected from below at $0$. It is
  well-known (see, e.g., Section 3 of \cite{whitt_reflmap}), 
  that, given $Y$,
  $(\Psi,\Phi)$ is the unique pair of paths with $\Psi$ increasing and $\Phi$
  positive such that $\Phi_t=Y_t+\Psi_t$ and the complementarity condition
  $\int_0^\infty \Indic{\Phi_t >0} \mathrm d \Psi_t = 0$ is satisfied.
  %%%%%%
  Using the transformations $\Upsilon,\Psi,\Phi$ we construct a pair of paths
  $(V^+,K^+)$ with $V^+$ positive and $K^+$ increasing and such that
  \eqref{taxcondition} and \eqref{compl_condition} are satisfied. It will be
  immediately clear from the construction that these properties are satisfied
  and that it is the only pair of paths satisfying these properties. 
  %%%%%%%%%%%%%%
  With $S_1=\inf\{t\geq 0:\Upsilon_t(X)<0\}$, we set $(V_t^+,K^+_t)=(\Upsilon_t(X),0)$, $t\in[0,S_1)$. If $S_1=\infty$ we are done.
  %%%%%%%%%%%%%%
  Otherwise, with $Y_t=V_{S_1-}^+ + \Delta X_{S_1} + X_{S_1+t}-X_{S_1}$ and $T_1=S_1+\inf\{t\geq 0:\Phi_t(Y)=\overline V^+_{S_1-}\}$ we set $(V_{S_1+t}^+,K_{S_1+t}^+ - K_{S_1-}^+))=(\Phi_t(Y),\Psi_t(Y))$, $t\in[0,T_1-S_1)$. If $T_1=\infty$ we are done. 
  %%%%%%%%%%%%%%
  Otherwise, with $Y_t=V^+_{T_1-}+ X_{T_1+t}-X_{T_1}$ (note $\Delta X_{T_1}=0$ by lack of upward jumps) and $S_2=T_1+\inf\{t\geq 0:\Upsilon_t(Y)<0\}$ we set $(V_{T_1+t}^+,K_{T_1+t}^+ - K_{T_1-}^+)=(\Upsilon_t(Y),0)$, $t\in[0,S_2-T_1)$. 
  %%%%%%%%%%%%%%
  In the obvious way we can continue this procedure to construct $(V^+,K^+)$ on $[S_2,T_2), [T_2,S_3),\ldots$, which yields the tax-reflection transform on the whole time horizon $[0,\infty)$ provided the time points $S_1,T_1,S_2,T_2,\ldots$ do not accumulate. To see that this is the case, we argue by contradiction and assume $T:=\lim_{n\to\infty} T_n<\infty$. Because we can identify, for  $n\geq 1$, $T_n=\inf\{t>S_n:V_t^+=\overline V^+_{S_n}\}$, we have $V^+_{T_n}\geq V_{0}=X_0>0$ while on the other hand $V^+_{S_n}=0$ for any $n\geq 1$. Hence $V^+$ does not have a right-limit at $T$. But \eqref{taxcondition} holds for all $t\in[0,T)$ and so the limit  $\lim_{t\uparrow T} V_t^+$ must exist because all three terms on the right-hand side of \eqref{taxcondition} have a right-limit at $T$ as $X$ is c\`adl\`ag on $[0,\infty)$ and $K^+$ and the integral term term are increasing on $[0,T)$. So we have reached a contradiction and we must have $T=\infty$ and hence $(V^+_t,K^+_t)$ is defined for all $t\geq 0$.
  
  Now assume $X_0\leq 0$ and $\bar{x}=0$. Then by \eqref{taxcondition}-\eqref{compl_condition} we must have $K_0^+=-X_0$ and so $(\overline{X+K^+})_t=\mathcal S_t(X+K^+)$ for all $t\geq 0$. Before tackling the general case we assume the special case where $\delta(x)=\gamma$ for all $x\geq 0$. By properties of the reflection map, $(V^+,K^+)$ is the tax-reflection transform of $X$ with constant tax rate $\gamma$ and initial maximum level $0$ if and only if 
  \begin{equation*}
    (V^+,K^+)= \left( \Phi \left( X-\gamma\mathcal S(X+K^+) \right) , \Psi \left( X-\gamma\mathcal S(X+K^+) \right) \right).
  \end{equation*}
  So if we can show there exists a unique c\`adl\`ag path $K^+$ to the fixed point equation
  \begin{equation}\label{fixedpointeq}
    K^+=\Psi \left( X-\gamma\mathcal S(X+K^+) \right),
  \end{equation} 
  then existence and uniqueness of the tax-reflection transform of $X$ with tax rate $\gamma$ and initial maximum level $0$ follows.
  To this end, let $T>0$. We claim that for $K_1,K_2$ two c\`adl\`ag paths,
  \begin{equation}\label{lipschitzprop}
    \sup_{t\in[0,T]} \left| \mathcal S_t(X+K^1)  - \mathcal S_t(X+K^2) \right| \leq \sup_{t\in[0,T]} \left| K_t^1 -K_t^2 \right|.
  \end{equation} 
  Indeed, suppose $t\in[0,T]$ is such that $\mathcal S_t(X+K^1) \geq \mathcal S_t(X+K^2)$. Then for any $\epsilon>0$ there exists $t'\in[0,t]$ such that $X_{t'}+K^1_{t'}\geq \mathcal S_t(X+K^1)-\epsilon$ and so since $\mathcal S_s(X+K^2)$ is increasing in $s$,
  \begin{equation*}
    \begin{split}
      \left|\mathcal S_t(X+K^1) - \mathcal S_t(X+K^2) \right|  %= & (\overline{X+K^1})_t -(\overline{X+K^2})_t \\
      \leq & X_{t'}+K^1_{t'} -\mathcal S_{t'}(X+K^2) +\epsilon \\
      \leq & X_{t'}+K^1_{t'} - (X_{t'}+K^2_{t'}) +\epsilon \\
      \leq & \sup_{s\in[0,T]} \left| K_s^1 -K_s^2 \right| +\epsilon.
    \end{split}
  \end{equation*}
  Since we can derive the same inequality for $t\in[0,T]$  such that $\mathcal S_t(X+K^1) \leq \mathcal S_t(X+K^2)$ for any given $\epsilon>0$, \eqref{lipschitzprop} follows. Similarly we can show for two c\`adl\`ag paths $Y^1$ and $Y^2$
  \begin{equation*}
    \sup_{t\in[0,T]} |\Psi_t(Y^1)-\Psi_t(Y^2)| \leq  \sup_{t\in[0,T]} |Y_t^1-Y_t^2|.
  \end{equation*}
  Hence for two c\`adl\`ag paths $K^1$ and $K^2$,
  \begin{multline*}
    \sup_{t\in[0,T]} \left| \Psi_t \left( X-\gamma\mathcal S(X+K^1) \right) -  \Psi_t \left( X-\gamma\mathcal S(X+K^2) \right) \right| \\
    \leq  \sup_{t\in[0,T]} \left|  X_t-\gamma\mathcal S_t(X+K^1) - (X_t-\gamma\mathcal S_t(X+K^2))  \right|
    \leq \gamma\sup_{t\in[0,T]} \left| K^1_t-K^2_t \right|.
  \end{multline*}
  Since $\gamma\in[0,1)$, we see that $K\mapsto \Psi \left( X-\gamma\mathcal S(X+K) \right)$ is a contraction mapping on the space of c\`adl\`ag functions on $[0,T]$ under the uniform metric, which is a complete metric space. So by the Banach fixed point theorem there exists, for each $T>0$, a unique c\`adl\`ag path $(K^+_t)_{t\in[0,T]}$ such that $K^+_t =\Psi_t \left( X-\gamma\mathcal S(X+K^+) \right)$ for $t\in[0,T]$. It follows that there exists a unique c\`adl\`ag  solution to the fixed point equation \eqref{fixedpointeq}.

  When $X_0\leq 0$, $\bar{x}=0$ and the tax rate $\delta$ is as in the statement of the proposition, then, with $\epsilon$ as in the statement of the proposition, $(V^+,K^+)$ on $[0,T_0)$ where $T_0=\inf\{t\geq 0: V_t^+=\epsilon\}$, is uniquely given by the tax-reflection transform of the path $X$ with constant tax rate $\gamma$ and initial maximum level $0$ and, if $T_0<\infty$, for $t\geq 0$, $(V^+_{T_0+t},K^+_{T_0+t}-K^+_{T_0-})$ is uniquely given by the tax-reflection transform of the path $s\mapsto \epsilon+X_{T_0+s}-X_{T_0}$ with tax rate $\delta$ and initial maximum level $\epsilon$.    

  It remains to cover the case where $\bar{x}>X_0\vee 0$. In this case one easily sees that  $(V^+,K^+)$ on $[0,T_1)$ is uniquely given by the reflection map $(\Phi(X),\Psi(X))$ with $T_1=\inf\{t\geq 0:\Phi_t(X)=\bar{x}\}$ and, if $T_1<\infty$, for $t\geq 0$, $(V^+_{T_1+t},K^+_{T_1+t}-K^+_{T_1-})$ is uniquely given by the tax-reflection transform of the path $s\mapsto \bar{x}+X_{T_1+s}-X_{T_1}$ with tax rate $\delta$ and initial maximum level $\bar{x}$.   
\end{proof}

\begin{remark}\label{remark_perturbed}
The case where $\delta(x)=\gamma$ for all $x\geq 0$ with $\gamma<1$  and the
path $X$ is continuous with $X_0=0$ in Theorem \ref{p:V} has been well-studied
in the literature. Namely, Le Gall and Yor \cite{legall_yor}\footnote{In
\cite{legall_yor}  $X$ is considered to be a Brownian motion, but their
argument equally applies to a general continuous path $X$.}  remarked at the
end of their paper that existence and uniqueness of the tax-reflection
transform can be obtained via a fixed point argument (as we did above) in the
case $\gamma\in(-1,1)$. Further, Davis \cite{davis96} (see also
\cite{davis_newyork}) shows that existence and uniqueness also hold in the
more difficult case where $\gamma=-1$ and that existence holds but uniqueness
fails when $\gamma<-1$. Despite that uniqueness does not hold in general when
$\gamma<-1$, Davis \cite{davis99} and Chaumont and Doney \cite{chaumount_doney}
were able to show that for $\gamma<-1$ there still exists a unique solution
(which is moreover adapted) to \eqref{taxcondition}-\eqref{compl_condition} if
$X$ is a Brownian motion. (Reflected) Brownian motions that are modified in
this way when they reach a new maximum are known in the literature as perturbed
Brownian motions. 
\end{remark}

\subsection{Construction of the natural tax process with bailouts}
\label{sec:construction}

\revised{In this section, we give a rigorous construction of the 
natural tax process with minimal bailouts $(V,K)$, which we described
briefly in Section \ref{sec:natural}.}
This is based on the tax-reflection transform (\cref{p:V}) 
of a spectrally negative Lévy process $X$.
Rather than construct it directly, we build a right-continuous
version $(V^+,K^+)$ and define $(V,K)$ in terms of this.

\begin{definition}
  \label{d:VK}
  Assume $\delta:[0,\infty)\to[0,1)$ satisfies the conditions of Theorem \ref{p:V}.
  Let $(V^+, K^+)$ be the pair of stochastic processes given by the tax-reflection
  transform of the spectrally negative Lévy process $X$ with $X_0=x$, tax rate $\delta$ and initial maximum level $\bar x\geq x\vee 0$ as defined in Theorem \ref{p:V}. 
  Then let $K_0=0$ and
  $K_t = K^+_{t-}$ for $t > 0$, to yield a left-continuous process $K$, and define $V_0=V_0^+-K_0^+$ and $V_t=V_t^+ - \Delta K_t^+$ for $t> 0$. 
%  \[
%    V_t = X_t + K_t - \int_{0^+}^t \delta(V_s) \, \dd (\overline{X+K})_s, \quad t \ge 0.
%  \]
  We call both $V$ and the pair $(V,K)$  the \emph{natural tax process with minimal bailouts and (natural) tax rate $\delta$} (under the
  measure $\PP_{x,\bar{x}}$). 
\end{definition}
The definition of $(V,K)$ implies
\begin{equation}\label{VK_eq}
	V_t = X_t + K_t - \int_{0^+}^t \delta(V_s) \, \dd (\overline{X+K})_s, \quad t\geq 0.
\end{equation}
Indeed, if $t>0$ is such that $\Delta K_t^+>0$, then $-\Delta X_t\geq \Delta K_t^+$ by  \eqref{taxcondition}-\eqref{compl_condition}. Similarly, if $K_0^+>0$, then $-X_0\geq K_0^+$. Therefore, $(\overline{X+K})_s=(\overline{X+K^+})_s$ for all $s\geq 0$ and the (countable) set $\{s\geq 0:V_s\neq V_s^+\}$ is a null set for the Lebesgue-Stieltjes measure $\mathrm d\mathrm (\overline{X+K^+})_s$. Hence \eqref{VK_eq} follows from \eqref{taxcondition}.

\begin{proposition}\label{prop_markov}
  \begin{enumerate} 	
    \item\label{item_VK_adapted}
      $(V,K)$ is adapted to the filtration $(\Ff_t)_{t\geq 0}$.
    \item\label{item_V_smp}
      The bivariate process $(V,\overline{V})$, with associated probability
      measures $\PP_{x,\bar{x}}$, is strong Markov in $(\Ff_t)_{t\geq 0}$.
  \end{enumerate}
\end{proposition}
\begin{proof}
 (i).     We will begin by showing that $(V^+, K^+)$, as defined in Definition \ref{d:VK}, is adapted to $(\Ff_t)_{t\ge 0}$.
      Considering the construction given in the proof of Theorem~\ref{p:V},
      if $Y$ is an adapted stochastic process, then the processes $\mathcal{S}(Y)$,
      $\Upsilon(Y)$, $\Phi(Y)$ and $\Psi(Y)$ defined there are adapted;
      for the process $\Upsilon(Y)$, this can be shown using Theorem 1(ii) (in particular
      Equation (6)) and Lemma 4 in \cite{ALW-tax-equiv}.
      Furthermore, the random times $S_1, T_1, S_2, T_2, \dotsc$ are stopping times with respect to $(\Ff_t)_{t\ge 0}$.
%%%%%%%%%%%
      This gives an immediate proof in the case that $x = \bar{x} > 0$.
      In the case where $x\le 0$, $\bar{x} = 0$ and $\delta(y) = \gamma$ for $y\geq 0$
      and $\gamma \in [0,1)$, we analyse the proof of Theorem~\ref{p:V} in more
      detail. Let
      $\Xi(L) = \Psi(X - \gamma \mathcal{S}(X+L))$.
      We have already shown that this is a contraction mapping, and evidently it
      maps adapted paths to adapted paths. The fixed point of $\Xi$ can be
      obtained by Picard iteration. Let $L^{(0)}_t = 0$ for all $t\ge 0$,
      and define $L^{(n)} = \Xi(L^{(n-1)})$ for $n\ge 1$.
      Then $K^+$ is the limit of $L^{(n)}$ as $n\to \infty$, and this limit
      holds uniformly on $[0,T]$ for any $T>0$. It follows that $K^+$ is
      adapted, and that the same holds for $V^+ = \Phi(X - \gamma\mathcal{S}(X+K^+))$.
      Finally, the remaining cases ($x\le 0$, $\bar{x} = 0$ and more general $\delta$;
      and $\bar{x} > x \vee 0$) can be resolved as in the proof of Theorem~\ref{p:V}.

      We have  shown that $(V^+, K^+)$ is adapted to $(\Ff_t)_{t\ge 0}$,
      and we turn now to $(V,K)$. Fix $t> 0$. 
      Since $K_t = \lim_{s\uparrow t} K^+_s$, it is $\Ff_t$-measurable, and so
      $K$ is adapted to $(\Ff_t)_{t\ge 0}$.
      Finally, the process $V$ satisfies \eqref{VK_eq}, $X$ and $K$ are adapted
      to $(\Ff_t)_{t\ge 0}$, and so is the integral 
      $t \mapsto \int_{0^+}^t \delta(V_s) \,\mathrm{d}(\overline{X+K})_s$,
      since it is continuous and adapted to $(\Ff_t)_{t\ge 0}$. 
      It follows that $V$ is adapted to $(\Ff_t)_{t\ge 0}$, and we are done.
   
\medskip

(ii).
  Let $T \ge 0$ be a finite stopping time with respect to $(\Ff_t)_{t\geq 0}$, and 
  let $(\tilde{V}^{+}, \tilde{K}^{+})$ be the tax-reflection
   transform of
\begin{equation*}
   \tilde{X}_t \coloneqq V_T + X_{T+t} - X_T, \quad t\ge 0, 
\end{equation*}
   with initial maximum level $\overline{V}_T$. 
 Now define
  \begin{align*}
  	\hat{V}^+_t &= 
  	\begin{cases}
  		{V}^{+}_t, & 0\le t< T, \\
  		\tilde{V}^{+}_{t-T}, &t\geq T,
  	\end{cases}
  	&\hat{K}^+_t &=
  	\begin{cases}
  		{K}^{+}_t, & 0\le t< T, \\
  		{K}^{+}_{T-} + \tilde{K}^{+}_{t-T}, &t\geq T,
  	\end{cases}
  \end{align*}
where $(V^+,K^+)$ is as in Definition \ref{d:VK}.
  Since $(\hat{V}^+,\hat{K}^+)$ is a tax-reflection transform of
  $X$ with initial maximum level $\bar{x}$, and such a transform is unique,
  it follows that $(V^+,K^+) = (\hat{V}^+, \hat{K}^+)$. Consequently, with $(\tilde V_0,\tilde K_0):=(\tilde V_0^+ - \tilde K_0^+,0)$ and $(\tilde V_t,\tilde K_t): = (\tilde V_t^+ - \Delta \tilde K^+_t,\tilde K^+_{t-})$ for $t>0$, we have $(V_{T+t})_{t\ge 0}=\tilde V$ and  by Definition \ref{d:VK} and \eqref{VK_eq},
  \begin{equation*}
	\tilde V_{t} = \tilde X_t + \tilde K_t - \int_{0^+}^t \delta(\tilde V_s) \, \dd (\overline{X+\tilde K})_s, \quad t\geq 0.
\end{equation*}
 The homogeneous strong Markov property of L\'evy processes \cite[Theorem~3.1]{Kyp2}
 means that, given $V_T$, $\tilde{X}$ has the law of $X$ under $\PP_{V_T,\overline{V}_T}$
  and is conditionally independent of $\Ff_T$.
  It follows that, given $(V_T,\overline{V}_T)$, the pair $(V_{T+t},\overline V_{T+t})_{t\geq 0}=(\tilde V_t,\sup_{0\leq s\leq t}\tilde V_s \vee \overline V_T)_{t\geq 0}$
  is conditionally independent of $\Ff_T$ and has the law of $(V,\overline V)$ under $\PP_{V_T,\overline{V}_T}$.
  This establishes the strong Markov property of $(V,\overline V)$.
% 
%    
%    \bigskip
%    
%    \textbf{Old proof:}
%      Now let $T \ge 0$ be a finite stopping time with respect to $(\Ff_t^+)_{t\geq 0}$, and 
%      define 
%      \[
%        \tilde{X}_t \coloneqq V_T + X_{t+T} - X_T, \quad t\ge 0,
%      \]
%      The homogeneous strong Markov property of Lévy processes \cite[Theorem~3.1]{Kyp2}
%      means that, given $V_T$, $(\tilde{X}_t)_{t\geq 0}$ has the law of $X$ under $\PP_{V_T}$
%      and is conditionally independent of $\Ff^+_T$.
%      As we have seen from the proof of part \ref{item_VK_adapted}\rl{But this is not covered in the new proof of part (i). Also have to be careful with the difference between $V$ and $V^+$ with the latter (together with $K^+$) but not the former being the tax-reflection transform.}, $(V_{T+t})_{t\ge 0}$ is the tax-reflection
%      transform of the process $\tilde{X}$
%      with initial maximum level $\overline{V}_T$.
%      It follows that, given $(V_T,\overline{V}_T)$, the pair $(V_{T+t},\overline V_{T+t})_{t\geq 0})$
%      is conditionally independent of $\Ff_T^+$ and has the law of $(V,\overline V)$ under $\PP_{V_T,\overline{V}_T}$.
%      This establishes the strong Markov property.
%  \end{enumerate}
\end{proof}

\subsection{Verification lemma and proof of \cref{p:value}}\label{sec_verif_char}

We state the following result, a generalisation of Lemma~2.1 in \cite{GeneralTaxStructure} and Lemma 4 in \cite{ALW-tax-equiv},
whose proof is however identical. 
%Though we make the assumption that $(H,L)$ is an admissible control, in fact, we only require the left-continuity of $H$ and $L$ together with conditions \ref{item_admis_bounds}--\ref{item_admis_cont} of Definition~\ref{def_admis}, and the proof is pathwise.
\begin{lemma}
  \label{l:sup}
  Fix $\bar x\geq 0$ and let $Y=(Y_t)_{t\geq 0}$ be a measurable real-valued path such that $\overline Y$ is continuous.
  Let $H=(H_t)_{t\geq 0}$ be a measurable path such that $H_t\leq 1$ for all $t\geq 0$. Then the path $U=(U_t)_{t\geq 0}$ defined by
$  %\begin{equation*}
  	U_t = Y_t - \int_{0^+}^t H_s \mathrm d \overline Y_s,
$  %\end{equation*}
 satisfies,
%\begin{equation*}
$  \overline{U}_t
= \overline{Y}_t - \int_{0+}^t H_s \, \dd \overline{Y}_s,
\  t\ge 0$.
%\end{equation*}
  
%  For any admissible control $(H,L) \in \Pi$\rl{TODO: Rewrite}, for any $\bar x\geq 0$ and $x\leq \bar x$, the controlled process
%  $U$ defined in \eqref{e:controlled-V} satisfies, under $\mathbb P_{x,\bar x}$,
%  \[
%    \overline{U}_t
%    = (\overline{X+L})_t - \int_{0+}^t H_s \, \dd (\overline{X+L})_s,
%    \quad t\ge 0.
%  \]
\end{lemma}

Our approach is based on the following verification lemma for the control problem \eqref{e:control}.
In order to state the conditions for this, we define the operator
\begin{equation*}
  \Aa_q g(x) = \frac{1}{2}\sigma^2 g''(x)   + c g'(x)    + \int_0^\infty \left( g(x-\theta) - g(x) + g'(x)\theta\Indic{\theta\leq 1} \right) \, \nu(\dd \theta) - q g(x) ,
\end{equation*}   
with the understanding that the first term on the right hand side disappears
when $\sigma=0$,
whenever $g$ has sufficient regularity and integrability such that the right-hand
side is well-defined and the integral against $\nu$ converges absolutely.

The lemma involves a candidate optimal value function $v(x,\bar{x})$, where $x$
and $\bar{x}$ have the role of the initial capital and initial maximum level,
respectively. We will write $\partial_x v$ for the derivative of $v$ in the first
variable, and $\partial_{\bar{x}} v$ for the derivative in the second variable.

\begin{lemma}[Verification lemma]
  \label{l:verif}
  Assume that $\int_1^\infty \theta \, \nu(\dd \theta) < \infty$ and $\eta\geq 0$.
  Let
  $v \from \RR\times [0,\infty) \to \RR$ be a function of the form 
  $v(x,\bar x)=\sum_{i=1}^p g_i(x)h_i(\bar x)$ for some $p\geq 1$,
  such that, for each $i=1,\ldots,p$:
  \begin{itemize}
    \item[(a)] $g_i$ is continuous on $\mathbb R$ and
      continuously differentiable on $\mathbb R\backslash\{0\}$ 
      with $g_i'(0):=\lim_{x\downarrow 0} g_i'(x)$ well-defined and real.
      If $X$ has paths of unbounded variation then, in addition, 
      $g_i$ is continuously differentiable on $\mathbb R$ and 
      twice continuously differentiable  on $(0,\infty)$. %Moreover, if $\sigma>0$ then $g_i''(0):=\lim_{x\downarrow 0} g_i''(x)$ is well-defined and real.
    \item[(b)]  $h_i$ is absolutely continuous on $[0,\infty)$ with a locally bounded, left-continuous density %of which a version is 
      denoted by $h'_i$.
  \end{itemize}
%   For a function $g:\mathbb R\to\mathbb R$ such that for some $\epsilon\geq 0$ the function $x\mapsto g(x-\epsilon)$ satisfies the smoothness properties in part (a) above and, for some $a,b\in\mathbb R$, $g(y)=ay+b$ for $y\leq -\epsilon$,  we denote by $\Aa_q g:[0,\infty)\to\mathbb R$ the function defined by, for $x\geq 0$,
%   ***
%   with the understanding that the first term on the right hand side disappears when $\sigma=0$. Note that $\Aa_q g(x)$ is well-defined because $\int_1^\infty \theta \, \nu(\dd \theta) < \infty$.
%   Write $\partial_x v(y,\bar y)=\sum_{i=1}^p g_i'(y)h_i(\bar y)$ and $\partial_{\bar{x}}v(y,\bar y)=\sum_{i=1}^p g_i(y)h'_i(\bar y)$.
  If $v$ satisfies in addition the conditions
  \begin{enumerate}[label=(\roman*),ref=(\roman*)]
    \item\label{itemverif_formbelow0}
      $v(x,\bar{x}) = \eta x + v(0,\bar{x})$
      for all $x < 0$ and $\bar{x}\geq 0$,
    \item\label{itemverif_genineq}
      $\Aa_q v(x,\bar{x}):= \Aa_q [v(\cdot , \bar x)](x) = 0$ for all $x>0$ and $\bar{x}\geq 0$, 

      % \fbox{Does not apply when $\sigma=0$ and $\int_0^1 \theta\nu(\mathrm d \theta)=\infty$ and $x=0$}
      %    \item
      %      $(\Aa - q)v(0,\bar{x}) = 0$
      %      for $0 < \bar{x}$,\rl{I think we should seriously
      %        look at the situation around zero with a view
      %        to unifying these conditions and improving the
      %      conclusion.}
    \item\label{itemverif_pdeineq}
      $\gamma \partial_x v(\bar{x},\bar{x})
      + (\gamma-1)\partial_{\bar{x}} v(\bar{x},\bar{x})
      \ge \gamma$ for all $\gamma\in [\alpha,\beta]$ and all
      $\bar{x}\geq 0$,

    \item\label{itemverif_bound} $v(x,\bar x)$ is bounded on the set $\{x\geq 0, \bar x\geq 0: x\leq \bar x\}$,

    \item\label{itemverif_derivbound}
      $\partial_x v(x,\bar{x}) \le \eta$ for all  $0\leq x\le \bar{x}$, 
  \end{enumerate}
  then $v(x,\bar{x}) \ge v^*(x,\bar{x})$ for all $(x,\bar x)\in\RR\times [0,\infty)$ such that $x\leq \bar x$.
\end{lemma}
\begin{proof}
  Before starting, we note that condition \ref{itemverif_genineq} makes sense.
  This is true because, for a function $g:\mathbb R\to\mathbb R$  satisfying the smoothness properties in part (a) above and such that, for some $a,b\in\mathbb R$, $g(y)=ay+b$ for $y\leq 0$,
  the function $\Aa_q g\colon (0,\infty)\to\mathbb{R}$ is well-defined. 
  This uses the assumption that
  $\int_1^\infty \theta \, \nu(\dd \theta) < \infty$.

  Fix $(x,\bar x)\in\RR\times [0,\infty)$ such that $x\leq \bar x$. We work on the filtered probability space $(\Omega,\mathcal F,(\Ff_t^x)_{t\geq 0},\mathbb P_{x,\bar x})$ where $\Ff_t^x$ is the completion of $\Ff_t$ under $\mathbb P_x$ so that the filtration $(\Ff_t^x)_{t\geq 0}$ satisfies the usual conditions.
  Fix $\pi:=(H,L)$ to be an admissible strategy,
  and write $U$ for the controlled process.
  We let $U^+$ and $L^+$ to be the right-continuous modifications of $U$ and $L$, respectively.
  Note that $U_t^+\geq 0$ for all $t\geq 0$ by Definition \ref{def_admis}\ref{item_admis_pos}, and $\overline U^+ \coloneqq \overline{U^+}$  has continuous sample paths  by \cref{l:sup} and Definition \ref{def_admis}\ref{item_admis_cont}.
  %Define $\overline{U}^+_t = \bar{x}\vee \sup_{s\le t} U_s$.
  %In order to take into account the possibility that $\partial_{xx} v(x,\bar{x})$, the second order partial derivative with respect to the first variable, is unbounded for $x$ in a neighbourhood of $0$, we fix $\epsilon>0$ and introduce the function $v_\epsilon(x,\bar x):=v(x+\epsilon,\bar x+\epsilon)$.  Define $T_n= \inf \left\lbrace t\geq 0: \widetilde{U}_{t} > n \right\rbrace$. 
  Fix $\epsilon>0$ and introduce the function $v_\epsilon\from \RR\times [0,\infty) \to \RR$ defined by $v_\epsilon(y,\bar y)=v(y+\epsilon,\bar y+\epsilon)$. We further denote $\partial_x v_\epsilon(y,\bar y):=\partial_x v(y+\epsilon,\bar y+\epsilon)$ and $\partial_{\bar{x}}v_\epsilon(y,\bar y):=\partial_{\bar{x}}v(y+\epsilon,\bar y+\epsilon)$.
  By the regularity assumptions, we can, for each $i=1,\ldots,p$, apply the Meyer-It\^o formula (Theorem IV.70 and Corollary 1 in \cite{Pro-si}) or the change of variables formula (see Theorem II.31 in \cite{Pro-si})\footnote{Note that the assumption of continuous differentiability in Theorem II.31 in \cite{Pro-si} can be relaxed to absolute continuity with a locally bounded density.} for $g_i(U^+ +\epsilon)$ depending on whether $X$ has paths of unbounded or bounded variation and we can use the change of variables formula for $h_i(\overline{U}^+ +\epsilon)$. Applying this together with the integration by parts formula (see p.68, Theorems II.26 and II.28 in \cite{Pro-si}) while noting that $h_i(\overline{U}^+ +\epsilon)$ 
  has continuous sample paths of bounded variation,  we get, for $t\geq 0$,
  \begin{equation*}
    \begin{split}
      e^{-q t}\, v_\epsilon & (U^+_{t},\overline{U}^+_{t}) 
      - v_\epsilon({U}^+_{0},\overline{U}^+_{0}) \\
      & = - \int^{t}_{0^{+}} q e^{-q s}\, v_\epsilon({U}^+_{s-},\overline{U}^+_{s})\, \dd s  
      \\
      & \quad {} +  \int^{t}_{0^{+}} e^{-q s}\sum_{i=1}^p \left( g_i({U}^+_{s-}+\epsilon)\mathrm d h_i(\overline{{U}}^+_{s}+\epsilon) + h_i(\overline{U}^+_{s-}+\epsilon)\mathrm d g_i(U^+_{s}+\epsilon) \right)  \\
      & =  -\int^{t}_{0^{+}} q e^{-q s}\, v_\epsilon({U}^+_{s-},\overline{U}^+ _{s})\, \dd s + \int^{t}_{0^{+}} e^{-q s}\, \partial_x v_\epsilon ({U}^+_{s-},\overline{U}^+_{s})\, \dd {U}^+_{s}\\
      & \quad {} + \int^{t}_{0^{+}} e^{-q s}\,  \partial_{\bar x} v_\epsilon ({U}^+_{s-}, \overline{U}^+_{s})\, \dd \overline{U}^+_{s} + \dfrac{1}{2} \int^{t}_{0^{+}} e^{-q s}  \partial_{xx} v_\epsilon ({U}^+_{s-}, \overline{U}^+_{s})\, \dd \left[ {U}^+ \right] ^{c}_{s}\\
      & \quad {} + \sum_{0 < s \leq t} e^{-q s} \left[ v_\epsilon({U}^+_{s}, \overline{U}^+_{s}) - v_\epsilon({U}^+_{s-}, \overline{U}^+_{s}) -  \partial_x v_\epsilon ({U}^+_{s-}, \overline{U}^+_{s})\, \Delta U^+_{s} \right],
    \end{split}	
  \end{equation*}
  where  $\partial_{xx}v(y,\bar y):=\sum_{i=1}^p g_i''(y+\epsilon)h_i(\bar y+\epsilon)$ and, for any process $Y$, %$\Delta Y_s:=Y_s - Y_{s-}$ and 
  $\left[ Y \right] ^{c}$ stands for the continuous part of the quadratic variation process of $Y$.
  %
  %The theorem is in Jacod and Shiryaev.
  %
  %Theorem II.1.8 says that for a random measure mu and compensator mu^p, for every predictable W such that |W|*mu is locally integrable and increasing,
  %W*mu - W*mu^p
  %is a local martingale.
  %
  %Corollary II.1.21 says that for any Poisson random measure (which they call "extended Poisson measure", the compensator is the intensity measure.
  %
  %The combination of these gives the result we require. (In fact, the proof of Theorem II.1.8, part "(i) => (ii)", is sufficient once we know the basic compensation formula.) 
  Since $[U^+]^{c}_{s} = [X]^{c}_{s}=\sigma^{2} s$ and 
  \begin{align*}
    v_\epsilon({U}^+_{s}, \overline{U}^+_{s}) - v_\epsilon({U}^+_{s-}, \overline{U}^+_{s}) &= v_\epsilon({U}^+_{s-} + \Delta X_{s} + \Delta {L}^+_{s},\overline{U}^+_{s})- v_\epsilon ({U}^+_{s{-}},\overline{U}^+_{s} )\nonumber\\
    &= v_\epsilon({U}^+_{s{-}} + \Delta X_{s} ,\overline{U}^+_{s})- v_\epsilon ({U}^+_{s{-}},\overline{U}^+_{s})\nonumber\\
    &\quad +  v_\epsilon({U}^+_{s{-}} + \Delta X_{s} + \Delta {L}^+_{s},\overline{U}^+_{s}) -v_\epsilon({U}^+_{s{-}} + \Delta X_{s} ,\overline{U}^+_{s}), 
  \end{align*}
  we get recalling the definition of the operator $\mathcal{A}_q$, \eqref{e:controlled-V} and using that $U_0^+=U_0+L^+_0$, $\overline U_0^+=\overline U_0$ and $\overline U^+_t=(\overline{X+L^+})_0 + \int_{0^+}^t (1-H_s) \mathrm d (\overline{X+L^+})_t$ from 
  \cref{l:sup},
  \begin{align}\label{eq_afterito}
    &e^{-q t}  v_\epsilon({U}^+_{t},\overline{U}^+ _{t}) -  v_\epsilon({U}_{0}, \overline{U}_{0}) \nonumber\\
    &\quad=  M_{t} + \int^{t}_{0^{+}}  e^{-q s}\,\mathcal{A}_q v_\epsilon({U}^+_{s-},\overline{U}^+_{s})\, \dd s  + \int^{t}_{0^{+}} e^{-q s}\, \partial_x v_\epsilon({U}^+_{s-}, \overline{U}^+_{s})\, \dd ({L}^+_{s})^c \nonumber\\
    &\qquad + \sum_{0 \leq s \leq t} e^{-q s} \Bigr\{  v_\epsilon({U}^+_{s{-}}+ \Delta X_{s} + \Delta {L}^+_{s}, \overline{U}^+_{s})-  v_\epsilon({U}^+_{s{-}} + \Delta X_{s}, \overline{U}^+_{s}) \Bigl\}\nonumber\\
    &\qquad - \int^{t}_{0^{+}} e^{-q s} \left[ \partial_x v_\epsilon(U^+_{s-}, \overline{U}^+_{s}) H_{s} -\partial_{\bar{x}} v_\epsilon(U^+_{s-}, \overline{U}^+_{s}) (1- H_{s}) \right] \dd \overline{(X+L^+)}_{s},
  \end{align}
  with the understanding that $U^+_{0-}:=U_0$, $\Delta X_{0}:=X_0$ and $\Delta L_0^+:=L_0^+$ and  where $(L_t^+)^c =L_t^+ - \sum_{0 < s \leq t} \Delta {L}^+_{s}$ and $M_{t}$ is the sum of $M_t^{(1)}$ and $M_{t}^{(2)}$  given by
  \begin{equation*}
    \begin{split}
     &  M_{t}^{(1)} = \int^{t}_{0^{+}} e^{-q s}\, \partial_x v_\epsilon({U}^+_{s-}, \overline{U}^+_{s})\, \dd \left[ X_{s} - c s - \sum_{0 < u \leq s} \Delta X_{u} \mathbf{1}_{\lbrace \vert \Delta X_{u} \vert > 1 \rbrace} \right], \\
     & M_{t}^{(2)} \\
     & =   \sum_{0 < s \leq t} e^{-q s} \Bigl\{v_\epsilon({U}^+_{s{-}} + \Delta X_{s}, \overline{U}^+_{s}) - v_\epsilon({U}^+_{s-}, \overline{U}^+_{s}) - \partial_x v_\epsilon({U}^+_{s-}, \overline{U}^+_{s})\,\Delta X_{s}\, \mathbf{1}_{\lbrace  \vert \Delta X_{s}\vert \leq 1 \rbrace}\Bigr\}\\
      & - \int^{t}_{0^{+}}  \int^{\infty}_{0^{+}} e^{-q s} \Bigl\{ v_\epsilon(U^+_{s-} - \theta, \overline{U}^+_{s}) - v_\epsilon(U^+_{s-},\overline{U}^+_{s}) +\theta\, \partial_x v_\epsilon({U}^+_{s-}, \overline{U}^+_{s})\,\mathbf{1}_{\lbrace  \theta \leq 1\rbrace}\Bigr\} \nu(\dd \theta)\,\dd s.
    \end{split}
  \end{equation*}
  Note that $M^{(1)}$ is a local martingale by the L{\'e}vy-It\^{o} decomposition (see e.g.~\cite[Chapter 2]{Kyp2}) and $M^{(2)}$ is a local martingale as it is an integral of a predictable process against a compensated Poisson random measure, see Theorem II.1.8 and Proposition II.1.21 in \cite{JS-limit}.
  %%%
  By the regularity assumptions and conditions \ref{itemverif_formbelow0} and \ref{itemverif_derivbound}, $y\mapsto v_\epsilon(y,\bar y)$ is absolutely continuous on $(-\infty,\bar y]$ with density $\partial_x v(y,\bar y)$ bounded by $\eta$, which means in particular
  \begin{equation*}
    v_\epsilon({U}^+_{s{-}}+ \Delta X_{s} + \Delta {L}^+_{s}, \overline{U}^+_{s})-  v_\epsilon({U}^+_{s{-}} + \Delta X_{s}, \overline{U}^+_{s}) \leq \eta \Delta {L}^+_{s},
    % v(x+y,\bar x) - v(x,\bar x) \leq \eta y \quad \text{for all $x\in\mathbb R$, $y,\bar x\geq 0$ such that $x+y\leq \bar x$.}
  \end{equation*}
  where we note that ${U}^+_{s{-}}+ \Delta X_{s} + \Delta {L}^+_{s}\leq  \overline{U}^+_{s}$ since otherwise $\overline{U}^+_s$ is not continuous in $s$.
  Using this in combination with conditions \ref{itemverif_genineq} (which implies $\mathcal{A}_q v_\epsilon(y,\bar y)=0$ for all $0\leq y\leq \bar y$ as $\epsilon>0$) and \ref{itemverif_pdeineq} while noting that  $\{s> 0 :U^+_{s-} < \overline{U}^+_{s}\}$ is a null set for the  Lebesgue–Stieltjes measure $\mathrm d(\overline{X+L^+})_{s}$ by 
  \cref{l:sup},
  we get from \eqref{eq_afterito}, for all $t\geq 0$,
  \begin{equation}\label{veriflem_ineq}
    \begin{split}
      e^{-q t}  v_\epsilon({U}^+_{t},\overline{U}^+ _{t})  
      %\leq & v_\epsilon({U}^+_{0}, \overline{U}_{0}^+)  +  M_{t}  + \eta\int^{t}_{0^+} e^{-q s} \dd {L}^+_{s}  
      %	- \int^{t}_{0^{+}} e^{-q s} H_s \dd \overline{(X+L^+)}_{s} \\
      & \leq  v_\epsilon({U}_{0}, \overline{U}_{0})  +  M_{t}  + \eta\int^{t}_{0} e^{-q s} \dd {L}^+_{s}  
      - \int^{t}_{0^{+}} e^{-q s} H_s \dd (\overline{X+L^+})_{s}.
    \end{split}
  \end{equation}	
  %where for the second inequality we used that $U_0^+=U_0+L^+_0$, $\overline U_0^+=\overline U_0$  and $\partial_x v_\epsilon(x,\bar x)\leq \eta$.
  Let $(T_n)_{n\geq 1}$ be a localising sequence for the local martingale $M$. Then $\mathbb E_{x,\bar x}[M_{t\wedge T_n}]=\mathbb E_{x,\bar x}[M_0]=0$ and so, %for any $\bar x\geq 0$ and $x\leq\bar x$,
  \begin{equation*}
    \begin{split}
      v_\epsilon(x,\bar x)
      & \geq \mathbb E_{x,\bar x} \left[ \int^{t\wedge T_n}_{0^{+}} e^{-q s} H_s \dd (\overline{X+L^+})_{s} \right]  -  \mathbb E_{x,\bar x} \left[ \eta\int^{t\wedge T_n}_{0} e^{-q s} \dd {L}^+_{s} \right]  \\
      & \quad {} +  \mathbb E_{x,\bar x} \left[ e^{-q (t\wedge T_n)}  v_\epsilon({U}^+_{t\wedge T_n},\overline{U}^+ _{t\wedge T_n}) \right].
    \end{split}
  \end{equation*}	
  Letting $t\to\infty$ and $n\to\infty$ we have by the monotone convergence theorem and the dominated convergence theorem in combination with condition \ref{itemverif_bound}, $U_t^+\geq 0$ and Definition \ref{def_admis}\ref{item_admis_expcap}, 
  \begin{equation*}
    v_\epsilon(x,\bar x)
    \geq \mathbb E_{x,\bar x} \left[ \int^{\infty}_{0^{+}} e^{-q s} H_s \dd (\overline{X+L^+})_{s} \right]  -  \mathbb E_{x,\bar x} \left[ \eta\int_{0}^{\infty} e^{-q s} \dd {L}^+_{s} \right]  = v^\pi(x,\bar x).
  \end{equation*}
  Taking $\epsilon\downarrow 0$ and using the joint continuity of $v(\cdot,\cdot)$ (which is implied by continuity of $g_i$ and $h_i$) yields $v(x,\bar x)\geq v^\pi(x,\bar x)$. % for any  $\bar x\geq 0$ and $x\leq\bar x$.
  %; and  the last equality comes from the fact that in the middle expression we can replace $L^+$ by $L$ as the integrands are continuous \fbox{EDIT}.
  Since $\pi$ and $(x,\bar x)$ were chosen arbitrarily the lemma follows.    
\end{proof}

\begin{lemma}[Characterisation of the value function]
  \label{l:char}
  Assume that $\int_1^\infty \theta \, \nu(\dd \theta) < \infty$ and $b,\eta\geq 0$.
  Let $a>b$. Let  $v_a\from \RR\times [0,\infty) \to \RR$ be a function satisfying the conditions of the function $v(\cdot,\cdot)$ in Lemma \ref{l:verif} except that conditions \ref{itemverif_pdeineq}-\ref{itemverif_derivbound} are replaced by the following:
  \begin{enumerate}[label=(\roman*'),ref=(\roman*'),start=3]
    %		\item\label{itemchar_geneq}
    %		$\Aa_q v_a(x,\bar{x}) = 0$ for $0\le x\le \bar{x}$, 
    %		
    %		 \fbox{Does not apply when $\sigma=0$ and $\int_0^1 \theta\nu(\mathrm d \theta)=\infty$ and $x=0$}
    %    \item
    %      $(\Aa - q)v(0,\bar{x}) = 0$
    %      for $0 < \bar{x}$,\rl{I think we should seriously
    %        look at the situation around zero with a view
    %        to unifying these conditions and improving the
    %      conclusion.}
    \item\label{itemchar_pdeeq} 
      \begin{align*}
        \alpha \partial_x v_a(\bar{x},\bar{x})
        + (\alpha-1)\partial_{\bar{x}} v_a(\bar{x},\bar{x})
        &= \alpha, &\bar{x}&\in[0,b], \\
        \beta \partial_x v_a(\bar{x},\bar{x})
        + (\beta-1) \partial_{\bar{x}} v_a(\bar{x},\bar{x})
        &= \beta, & \bar{x}&\in (b,a],
      \end{align*}

    \item\label{itemchar_valueata} $v_a(a,a)=0$.

      %	\item\label{itemverif_derivbound}
      %	$\partial_x v_a(x,\bar{x}) \le \eta$ for $0\leq x\le \bar{x}$, 
  \end{enumerate}
  Then for any $\bar x\in[0,a]$ and $x\leq \bar x$,
  \begin{equation*}
    v_a(x,\bar{x})
    =
    \EE_{x,\bar{x}}
    \left[
      \int_{0^+}^{\tau_a^+} e^{-qs} \delta_b(V_s) \, \dd (\overline{X+K})_s
      -
      \eta
      \int_0^{\tau_a^+} e^{-qs} \, \dd K^+_s
    \right],
  \end{equation*}
  where $(V,K)$ is the natural tax process with minimal bailouts and threshold tax rate $\delta_b$, and $\tau_a^+=\inf\{t\geq 0:V_t\geq a\}$.
\end{lemma}
\begin{proof}[Proof of \cref{l:char} (characterisation of the value function)]
  Fix $a>b$. We fix $\epsilon>0$ and let $v_{a,\epsilon}\from \RR\times [0,\infty) \to \RR$ be the function defined by $v_{a,\epsilon}(x,\bar x)=v_a(x+\epsilon,\bar x)$, though we purposely do not also shift the second argument by $\epsilon$ as we did  in the proof of Lemma \ref{l:verif}. Because $\pi=(\delta_b(V),K)$ satisfies the conditions in Definition \ref{def_admis} except (possibly) for item \ref{item_admis_expcap} (recall Definition \ref{d:VK} and adaptedness follows from Proposition \ref{prop_markov}\ref{item_VK_adapted}) and $v_a$ satisfies the regularity conditions of $v$ in Lemma \ref{l:verif}, we deduce, by following the proof of Lemma \ref{l:verif}, the identity \eqref{eq_afterito} but with $U$ replaced by $V$, $L$ replaced by $K$, $v_\epsilon$ by $v_{a,\epsilon}$ and $H$ by $\delta_b(V)$. 
  %By definition of $(V^+,K^+)$, $V_{s-}^+ + \Delta X_s + \Delta K_s^+ = 0$ if $\Delta K_s^+>0$ and the support of $\mathrm d K_s^+$ is contained in $\{s\geq 0: U^+_{s-}=0\}$. Hence  $v_a({V}^+_{s{-}}+ \Delta X_{s} + \Delta {K}^+_{s}, \overline{V}^+_{s})-  v_a({V}^+_{s{-}} + \Delta X_{s}, \overline{V}^+_{s})=\eta \Delta K_s^+$ by \ref{itemverif_formbelow0}  and note that $\partial_x v_a(0,\bar x)=\eta$ \fbox{IS THIS TRUE?} by the regularity assumptions and condition \ref{itemverif_formbelow0}. Then using the arguments leading to \eqref{veriflem_ineq} but replacing the inequalities \ref{itemverif_genineq}, \ref{itemverif_pdeineq} and \ref{itemverif_derivbound} by the equalities \ref{itemchar_geneq} and \ref{itemchar_pdeeq}, we get
  %  \begin{equation}\label{charlem_eq}
  %	e^{-q t}  v_a({V}^+_{t},\overline{V}^+ _{t})  
  %	=  v_a({V}_{0}, \overline{V}_{0})  +  M_{t}  + \eta\int^{t}_{0} e^{-q s} \dd {K}^+_{s}  
  %	- \int^{t}_{0^{+}} e^{-q s} \delta(V_s) \dd \overline{(X+K^+)}_{s}.
  %\end{equation}	
  %%%
  Then replacing $t$ by $\tau_a^+\wedge t\wedge T_n$ with $(T_n)_{n\geq 1}$ a localising sequence for the local martingale $(M_t)_{t\geq 0}$ and taking expectations as well as using condition \ref{itemverif_genineq} in Lemma \ref{l:verif}, we get, for $\bar x\geq 0$ and $x\leq \bar x$,  
  \begin{equation}\label{eq_charlem_witheps}
    \begin{split}
      %\EE&_{x,\bar{x}}  \left[ v_{a,\epsilon}({V}^+_{0}, \overline{V}^+_{0})  \right] \\
   &   v_{a,\epsilon}(x,\bar x)  =  \EE_{x,\bar{x}} \Bigg[ \int^{\tau_a^+\wedge t\wedge T_n}_{0^{+}} e^{-q s} \\
   & \times \left[ \partial_x v_{a,\epsilon}(V^+_{s-}, \overline{V}^+_{s}) \delta_b(V_{s}) -\partial_{\bar{x}} v_{a,\epsilon}(V^+_{s-}, \overline{V}^+_{s}) (1- \delta_b(V_{s})) \right] \dd (\overline{X+K^+})_{s}    \\
        & -  \sum_{0 \leq s \leq \tau_a^+\wedge t\wedge T_n} e^{-q s} \Bigr\{  v_{a,\epsilon}({V}^+_{s{-}}+ \Delta X_{s} + \Delta {K}^+_{s}, \overline{V}^+_{s})-  v_{a,\epsilon}({V}^+_{s{-}} + \Delta X_{s}, \overline{V}^+_{s}) \Bigl\}   \\
        &   -  \int^{\tau_a^+\wedge t\wedge T_n}_{0^{+}} e^{-q s}\, \partial_x v_{a,\epsilon}({V}^+_{s-}, \overline{V}^+_{s})\, \dd ({K}^+_{s})^c   
      + 	e^{-q (\tau_a^+\wedge t\wedge T_n) }  v_{a,\epsilon}({V}^+_{\tau_a^+\wedge t\wedge T_n},\overline{V}^+ _{\tau_a^+\wedge t\wedge T_n}) \Bigg].
    \end{split}
  \end{equation}
  %with the understanding that $V^+_{0-}=V_0$ and $X_{0-}=X_0$. 
  By the smoothness assumptions, $v_a(\cdot,\cdot)$, $\partial_x v_a(\cdot,\cdot)$ and $\partial_{\bar x}v_a(\cdot,\cdot)$ are all bounded on $[0,a+1]\times [0,a]$ and right-continuous in the first argument.
  By the mean value theorem, for all $\epsilon\in(0,1)$,
  \begin{multline*}
    \sum_{0 \leq s \leq \tau_a^+\wedge t\wedge T_n} e^{-q s} \Bigr|  v_{a,\epsilon}({V}^+_{s{-}}+ \Delta X_{s} + \Delta {K}^+_{s}, \overline{V}^+_{s})-  v_{a,\epsilon}({V}^+_{s{-}} + \Delta X_{s}, \overline{V}^+_{s}) \Bigl| \\
    \leq \sup_{y\in[0,a+1],z\in[0,a]} \partial_x v_{a}(y,z)\sum_{0 \leq s \leq \tau_a^+\wedge t\wedge T_n} e^{-q s}  \Delta {K}^+_{s}.
  \end{multline*}
  Further, $a\geq \overline V^+_{t\wedge \tau_a^+} \geq (1-\beta)(\overline{X+K^+})_{t\wedge \tau_a^+}\geq   (1-\beta)K^+_{t\wedge \tau_a^+}$ where the second inequality is due to 
  \cref{l:sup}.
  From these observations it is justified to apply the dominated convergence
  and take $\epsilon\downarrow 0$ in \eqref{eq_charlem_witheps} to conclude
  that \eqref{eq_charlem_witheps} holds with $\epsilon=0$.  By definition of
  $(V^+,K^+)$, $V_{s-}^+ + \Delta X_s + \Delta K_s^+ = 0$ if $\Delta K_s^+>0$.
  Hence  $v_a({V}^+_{s{-}}+ \Delta X_{s} + \Delta {K}^+_{s},
  \overline{V}^+_{s})-  v_a({V}^+_{s{-}} + \Delta X_{s},
  \overline{V}^+_{s})=\eta \Delta K_s^+$ by condition
  \ref{itemverif_formbelow0} of Lemma \ref{l:verif}. Further, $\{s> 0:
  V^+_{s-}>0\}$ is a null set for the Lebesgue-Stieltjes measure $\mathrm d
  (K_s^+)^c$. To see this, note that $\{s> 0:
  V^+_{s-}>0\}\subset\{s>0:V_s^+>0\}\cup\{s>0:\Delta V_s^+>0\}$ and
  $\{s>0:V_s^+>0\}$ is a null set for $\mathrm d K_s^+$ by
  \eqref{compl_condition}
  and $\{s>0:\Delta V_s^+>0\}$ is countable (because a
  c\`adl\`ag path has a countable number of jumps) and so is a null set for the
  continuous measure $\mathrm d (K_s^+)^c$. 
  %Lemma 5.8, Thm 5.9(i), Exercise 6.5, Thm 5.6
  We also note that $\partial_x v_a(0,\bar x)=\eta$ if the spectrally negative L\'evy process $X$ has paths of unbounded variation by continuity of $\partial_x v_a$ in that case, whereas $({K}^+_{s})^c=0$ for all $s\geq 0$ if $X$ has paths of bounded variation. Using these observations in combination with \eqref{eq_charlem_witheps} for $\epsilon=0$ and \ref{itemchar_pdeeq} (recall from the proof of Lemma \ref{l:verif} that  $\{s> 0 :V^+_{s-} < \overline{V}^+_{s}\}$ is a null set for the measure $\mathrm d(\overline{X+K^+})_{s}$) yields
  \begin{multline*}
    v_{a}(x,\bar x) 
    =  \EE_{x,\bar{x}}\Bigg[  \int^{\tau_a^+\wedge t\wedge T_n}_{0^{+}} e^{-q s} \delta_b(V_{s})\dd (\overline{X+K^+})_{s}  
      - \eta  \int^{\tau_a^+\wedge t\wedge T_n}_{0} e^{-q s}\,  \dd {K}^+_{s}  \\
    {} + e^{-q (\tau_a^+\wedge t\wedge T_n) }  v_{a}({V}^+_{\tau_a^+\wedge t\wedge T_n},\overline{V}^+ _{\tau_a^+\wedge t\wedge T_n}) \Bigg].
  \end{multline*}
  %%%%%%%
  With $(\Psi,\Phi)$ the reflection map from the proof of Theorem \ref{p:V},
  \begin{equation*}
    \overline V^+_t \geq (1-\beta)(\overline{X+K^+})_t\geq (1-\beta)(\overline{X+\Psi_t(X)}) =(1-\beta)\overline{\Phi_t(X)}, \quad t\geq 0,
  \end{equation*}
  where the second inequality is due to $K^+_t\geq \Psi_t(X)$ since $X_t\geq V_t^+$ and the equality is by definition of the reflection map. 
  Because $\overline{\Phi_t(X)}\to\infty$ as $t\to\infty$, %as can be proved by taking $q=0$ in \cite[Proposition~2]{Pist-exit},
   it follows that $\tau_a^+<\infty$. Hence $\tau_a^+\wedge t\wedge T_n\to\tau_a^+$ as $n,t\to\infty$.
  So letting $t\to\infty$ and $n\to\infty$, we get the desired identity by the dominated convergence theorem and  using condition \ref{itemchar_valueata} in combination with the process $V^+$ not having upward jumps. 
\end{proof}

%\section{}

\phantomsection
\label{s:containing-p:value}

\begin{proof}[Proof of \cref{p:value}] % (value function for the threshold tax rate with minimal bailouts)]
  %%%%%%%%%%%
  For $\gamma \in [0,1)$, $a>0$ and $x\geq 0$, define
  \begin{equation*}
    R_{\gamma,a}(x) = 
    \begin{cases}
      \frac\gamma{1-\gamma} Z^{(q)}(x)^{\frac1{1-\gamma}}\int_x^a Z^{(q)}(y)^{-\frac1{1-\gamma}}(1-\eta Z^{(q)}(y))\mathrm d y & \text{if $\gamma\in(0,1)$ and $x\leq a$}, \\
      0 & \text{if $\gamma=0$ or $x>a$}.
    \end{cases}
  \end{equation*}
  Fix $a\in (b,\infty)$ and define $v_a\from \RR\times[0,\infty) \to \RR$ by
  \begin{equation}\label{guess}
    \begin{split}
      v_a(x,\bar x)
      & = \eta \left( \overline Z^{(q)}(x)+\frac{\psi'(0)}q \right)   \\
      &\quad {} + \frac{Z^{(q)}(x)}{Z^{(q)}(\bar x)} \left\{ R_{\alpha,a}(\bar x) + \left(  \frac{Z^{(q)}(\bar x)}{Z^{(q)}(\bar x\vee b)} \right)^{\frac1{1-\alpha}} ( R_{\beta,a}(\bar x \vee b) - R_{\alpha,a}(\bar x\vee b) )  \right\} \\
      &\quad {} - \eta \frac{Z^{(q)}(x)}{Z^{(q)}(a)} \left( \overline Z^{(q)}(a)+\frac{\psi'(0)}q \right) \left(  \frac{Z^{(q)}(\bar x)}{Z^{(q)}(\bar x\vee b)} \right)^{\frac\alpha{1-\alpha}} \left(  \frac{Z^{(q)}(\bar x\vee b)}{Z^{(q)}(a)} \right)^{\frac\beta{1-\beta}}.
    \end{split}
  \end{equation}
  We claim that for all $(x,\bar x)\in\RR\times [0,\infty)$ such that $x\leq \bar x\leq a$,
  \begin{equation}\label{claim_char_stopata}
    v_a(x,\bar x)
    =  \EE_{x,\bar{x}}
    \left[
      \int_0^{\tau_a^+} e^{-qs} \delta_b(V_s) \, \dd (\overline{X+K})_s
      -
      \eta
      \int_0^{\tau_a^+} e^{-qs} \, \dd K^+_s
    \right],
  \end{equation}
  see also Remark \ref{remark_guess} below.
  Once \eqref{claim_char_stopata} is shown, \cref{p:value} follows, first for the $\eta=0$ case and then for the general case, by taking $a\to\infty$  and using the monotone convergence theorem and l'H\^opital's rule in combination with \eqref{limit_ratio_scale} (recall the assumption $\beta>0$).
  To prove \eqref{claim_char_stopata} we show that $v_a$ satisfies the conditions of \cref{l:char}. Regarding the smoothness conditions we have that $v_a$ is of the form $\sum_{i=1}^2 g_i(x)h_i(\bar x)$ with $g_1(x)=Z^{(q)}(x)$, $g_2(x)=\eta \left( \overline{Z}^{(q)}(x)+ \frac{\psi'(0)}{q} \right)$, $h_2(\bar x)=1$ and
  \begin{equation*}
    \begin{split}
      h_1(\bar x)
      &= \frac1{Z^{(q)}(\bar{x})} \left\{ R_{\alpha,a}(\bar x) + \left(  \frac{Z^{(q)}(\bar x)}{Z^{(q)}(\bar x\vee b)} \right)^{\frac1{1-\alpha}} ( R_{\beta,a}(\bar x \vee b) - R_{\alpha,a}(\bar x\vee b) )  \right\} \\
      & \quad {} -   \frac{\eta}{Z^{(q)}(a)} \left( \overline Z^{(q)}(a)+\frac{\psi'(0)}q \right) \left(  \frac{Z^{(q)}(\bar x)}{Z^{(q)}(\bar x\vee b)} \right)^{\frac\alpha{1-\alpha}} \left(  \frac{Z^{(q)}(\bar x\vee b)}{Z^{(q)}(a)} \right)^{\frac\beta{1-\beta}}.
    \end{split}
  \end{equation*}
  Recall the regularity properties of $W^{(q)}$ mentioned in Section \ref{sec_levy}. They imply that $g_1$ satisfies the required smoothness conditions in \cref{l:char} and then obviously so does $g_2$. Further, they imply that $h_1$ is continuous on $[0,\infty)$ and continuously differentiable  on $[0,\infty)\backslash\{b,a\}$ with the left-derivative serving as the left-continuous density $h_1'$ on $[0,\infty)$. Therefore, $h_1$ satisfies the required smoothness conditions in \cref{l:char} as well. It is easy to see that $v_a$ satisfies condition \ref{itemverif_formbelow0} in \cref{l:verif} by definition of $Z^{(q)}$ and $\overline Z^{(q)}$. Next, it is well-known that $\Aa_q Z^{(q)}(x)=0$ and  $\Aa_q\overline Z^{(q)}(x)=\psi'(0)$ for all $x>0$. 
  This can be deduced by using martingale arguments and stochastic calculus
  (see, for instance, (3.7), Proposition 2 and Lemma 5 in \cite{APP2008});
  or, alternatively, by taking Laplace transforms and using Fubini's theorem.
  %For $x=0$ one can easily check that $\Aa_q\overline Z^{(q)}(0)=\psi'(0)$ and 
  %\begin{equation*}
  %	\Aa_q Z^{(q)}(0)=
  %	\begin{cases}
  %		0 & \text{if $\int_0^1\theta\nu(\mathrm d\theta)<\infty$ or $\sigma>0$}, \\
  %		-q  & \text{if $\int_0^1\theta\nu(\mathrm d\theta)=\infty$ and $\sigma=0$}.
  %	\end{cases}
  %\end{equation*}
  So \ref{itemverif_genineq} in Lemma \ref{l:verif} is satisfied as well. 
  It is left to show that $v_a$ satisfies conditions \ref{itemchar_pdeeq} and \ref{itemchar_valueata} in \cref{l:char}. One easily sees that \ref{itemchar_valueata} holds. For \ref{itemchar_pdeeq}, by
  observing that
  \begin{equation}\label{Rgamderiv}
    R_{\gamma,a}'(x) = \frac1{1-\gamma}\frac{Z^{(q)\prime}(x)}{Z^{(q)}(x)}   R_{\gamma,a}(x) - \frac\gamma{1-\gamma}  ( 1-\eta Z^{(q)}(x) )  , \quad x\in[0,a),
  \end{equation}
  we get for $\bar x\in(b,a]$,
  \begin{equation}\label{derivcond_aboveb}
    \begin{split}
      \gamma \partial_x & v_a(\bar{x},\bar{x}) 
      + (\gamma-1)\partial_{\bar{x}} v_a(\bar{x},\bar{x}) \\
      & = \gamma Z^{(q)\prime}(\bar x) h_1(\bar x) + \gamma \eta Z^{(q)}(\bar x) + (\gamma-1)Z^{(q)}(\bar x) h_1'(\bar{x}) \\
      & =  R_{\beta,a}(\bar x)\frac{Z^{(q)\prime}(\bar x)}{Z^{(q)}(\bar x)} \left( 1-\frac{\gamma-1}{\beta-1} \right)  +\beta\frac{\gamma-1}{\beta-1} + \eta Z^{(q)}(\bar x) \left( \gamma - \beta\frac{\gamma-1}{\beta-1} \right) \\
      & \quad {} - \eta \frac{Z^{(q)\prime}(\bar x)}{Z^{(q)}(a)} \left( \overline Z^{(q)}(a)+\frac{\psi'(0)}q \right) \left(  \frac{Z^{(q)}(\bar x)}{Z^{(q)}(a)} \right)^{\frac\beta{1-\beta}} \left( \gamma  + \frac{(\gamma-1)\beta}{1-\beta}  \right) .
    \end{split}
  \end{equation}
  Similarly, for $0\leq \bar x\leq b$,
  \begin{equation}\label{derivcond_belowb}
    \begin{split}
     & \gamma \partial_x  v_a(\bar{x},\bar{x})
      + (\gamma-1)\partial_{\bar{x}} v_a(\bar{x},\bar{x}) \\
      %= & \gamma qW^{(q)}(\bar x) h_1(\bar x) + \gamma \eta Z^{(q)}(\bar x) + (\gamma-1)Z^{(q)}(\bar x) h_1'(\bar{x}) \\
      & = \left( R_{\alpha,a}(\bar x) + \left(  \frac{Z^{(q)}(\bar x)}{Z^{(q)}(b)} \right)^{\frac1{1-\alpha}} ( R_{\beta,a}(b) - R_{\alpha,a}(b) )  \right) \frac{Z^{(q)\prime}(\bar x)}{Z^{(q)}(\bar x)} \left( 1-\frac{\gamma-1}{\alpha-1} \right) +\alpha\frac{\gamma-1}{\alpha-1} \\
      & \quad {} + \eta  Z^{(q)}(\bar x)  \left( \gamma  - \alpha\frac{\gamma-1}{\alpha-1}   \right) \\
      & \quad {} - \eta \frac{Z^{(q)\prime}(\bar x)}{Z^{(q)}(a)} \left( \overline Z^{(q)}(a)+\frac{\psi'(0)}q \right)  \left(  \frac{Z^{(q)}(b)}{Z^{(q)}(a)} \right)^{\frac\beta{1-\beta}} \left(  \frac{Z^{(q)}(\bar x)}{Z^{(q)}(b)} \right)^{\frac\alpha{1-\alpha}} \left( \gamma  + \frac{(\gamma-1)\alpha}{1-\alpha}  \right) .
    \end{split}
  \end{equation}
  From \eqref{derivcond_aboveb} and \eqref{derivcond_belowb} we immediately see that condition \ref{itemchar_pdeeq} in \cref{l:char} is satisfied which completes the proof.
\end{proof}
\begin{remark}\label{remark_guess}
  We explain here briefly (details can be found in \cite{AlGhanim-thesis}) how we arrived at the guess \eqref{guess} as the (correct) expression for the expectation in \eqref{claim_char_stopata}. To this end, denote, for $\bar x\in[0,a]$ and $x\leq 0$, by $\widetilde v_a(x,\bar x)$ the right hand side of \eqref{claim_char_stopata}.  First, use known fluctuation identities for spectrally negative L\'evy processes reflected at their infimum to derive an expression for $\widetilde v_a(x,\bar x)$ in terms of scale functions and $\widetilde v_a(\bar x,\bar x)$. Second, use this expression to compute $\partial_x \widetilde v_a(\bar x,\bar x)$ and plug it into the equation \ref{itemchar_pdeeq} in Lemma \ref{l:char}. This yields then an ordinary differential equation for $\widetilde v_a(\bar x,\bar x)$ which, together with the boundary condition \ref{itemchar_valueata} in Lemma \ref{l:char}, can be uniquely solved. This will then yield the guess \eqref{guess}  for $\widetilde v_a(x,\bar x)$. Note that this line of arguments does not constitute a proof of \eqref{claim_char_stopata} because this would involve showing in addition that $\widetilde v_a(x,\bar x)$ satisfies,  a priori,  \ref{itemchar_pdeeq} in Lemma \ref{l:char}, which is non-trivial.
\end{remark}	

\subsection{Proof of \cref{t:control}}\label{sec_proof_mainthm}

%%%%%%%
Let $v:\mathbb R\times[0,\infty)\to\mathbb R$ be such that $v(x,\bar x)$ is equal to the right hand side of \eqref{e:vtaxinj} for all $x\in\mathbb R$ and $\bar x\geq 0$ and with $b=b^*$. Clearly $v(x,\bar x)\leq v^*(x,\bar x)$ for $x\leq \bar x$ because, for $x\leq \bar x$, $v$ is the value function of the control $(\delta_{b^*}(V),K)$ which is admissible (in particular Definition \ref{def_admis}\ref{item_admis_expcap} is implied by Proposition \ref{p:value}). So Theorem \ref{t:control} is proved once we show that $v$ satisfies all conditions of the verification lemma (Lemma \ref{l:verif}). Following the arguments in the proof of Proposition \ref{p:value} we see that the smoothness conditions as well as conditions \ref{itemverif_formbelow0} and \ref{itemverif_genineq} of Lemma \ref{l:verif} are satisfied. It remains to show that conditions \ref{itemverif_pdeineq}, \ref{itemverif_bound} and \ref{itemverif_derivbound} of Lemma \ref{l:verif} are satisfied. We start with condition \ref{itemverif_pdeineq}.
From \eqref{derivcond_aboveb} with $a\to\infty$, for $\bar x> b^*$ and  $\gamma\in[\alpha,\beta]$,
\begin{equation*}
  \gamma \partial_x v(\bar{x},\bar{x}) 
  + (\gamma-1)\partial_{\bar{x}} v(\bar{x},\bar{x}) 
  - \gamma
  %= & \left( \eta Z^{(q)}(\bar x) + R_\beta(\bar x)\frac{Z^{(q)\prime}(\bar x)}{Z^{(q)}(\bar x)} \right)   \frac{\beta-\gamma}{\beta-1}    +\beta\frac{\gamma-1}{\beta-1} -\gamma \\
  =  \biggl( \eta Z^{(q)}(\bar x) + R_\beta(\bar x)\frac{Z^{(q)\prime}(\bar x)}{Z^{(q)}(\bar x)} - 1\biggr)   \frac{\beta-\gamma}{\beta-1}
  \geq  0,
\end{equation*}
where the inequality is due to $C(\bar x)\leq Q(\bar x)$ for $\bar x> b^*$, see Lemma \ref{lem_CQ}, which implies $R_\beta(\bar x)\leq \frac{Z^{(q)}(\bar x)}{Z^{(q)\prime}(\bar x)} (1-\eta Z^{(q)}(\bar x))$ in combination with $\frac{\beta-\gamma}{\beta-1}\leq 0$ for all $\gamma\in[\alpha,\beta]$. On the other hand,  by \eqref{derivcond_belowb} (taking $a\to\infty$), we have  for $\bar x\in[0,b^*]$ and  $\gamma\in[\alpha,\beta]$,
\begin{equation*}
  \begin{split}
   &  \gamma   \partial_x v(\bar{x},\bar{x})  + (\gamma-1)\partial_{\bar{x}} v(\bar{x},\bar{x}) -\gamma \\
    %= & \gamma qW^{(q)}(\bar x) h_1(\bar x) + \gamma \eta Z^{(q)}(\bar x) + (\gamma-1)Z^{(q)}(\bar x) h_1'(\bar{x}) \\
    & = \left( \eta  Z^{(q)}(\bar x)  + \left\{ R_\alpha(\bar x) + \left(  \frac{Z^{(q)}(\bar x)}{Z^{(q)}(b^*)} \right)^{\frac1{1-\alpha}} ( R_\beta(b^*) - R_\alpha(b^*) )  \right\}  \frac{Z^{(q)\prime}(\bar x)}{Z^{(q)}(\bar x)}  - 1 \right)  \frac{\alpha-\gamma}{\alpha-1} \\
    & =  \left( \eta  Z^{(q)}(\bar x)  + \left\{ R_\alpha(\bar x) + Z^{(q)}(\bar x)^{\frac1{1-\alpha}} C(b^*) \right\}  \frac{Z^{(q)\prime}(\bar x)}{Z^{(q)}(\bar x)}  - 1 \right)  \frac{\alpha-\gamma}{\alpha-1} \\
    & \geq 0 ,
  \end{split}
\end{equation*}
where the inequality is due to $C(b^*)\geq C(\bar x)\geq Q(\bar x)$ for $\bar x\in[0,b^*]$, see Lemma \ref{lem_CQ}, in combination with $\frac{\alpha-\gamma}{\alpha-1}\geq 0$  for all $\gamma\in[\alpha,\beta]$. 

%\begin{remark}
%  Note that 
%  \begin{equation*}
%    \gamma \partial_x v(\bar{x},\bar{x}) + (\gamma-1)\partial_{\bar{x}} v(\bar{x},\bar{x}) -\gamma
%    = 
%    \begin{cases}
%      (\partial_x v(\bar{x},\bar{x}) - 1 ) \frac{\beta-\gamma}{\beta-1}   & \text{if $\bar x> b^*$}, \\[1.5ex]
%      (\partial_x v(\bar{x},\bar{x}) - 1 ) \frac{\alpha-\gamma}{\alpha-1}  & \text{if $\bar x\in[0,b^*]$}. \\
%    \end{cases}
%  \end{equation*}
%  and hence %the key inequality,  namely 
%  condition (iii) in Lemma \ref{l:verif} is equivalent to $\partial_x v(\bar{x},\bar{x})\geq 1$ for $\bar x\in[0,b^*]$ and $\partial_x v(\bar{x},\bar{x})\leq 1$ for $\bar x> b^*$. This is reminiscent of the key inequality for showing optimality of the refraction strategy in the optimal dividends problem  with control strategies that are absolutely continuous with a density bounded by a given constant, see Lemma 6 in \cite{OptimalControl2012}.
%\end{remark}	

Next we look at condition \ref{itemverif_bound} of Lemma \ref{l:verif}.
Let us denote $v(x,\bar x)=v^{(\eta)}(x,\bar x)$ to indicate the dependence of $v$ on $\eta$. Further let $D=\{x\geq 0, \bar x\geq 0: x\leq \bar x\}$. We first deal with the case $\eta=0$. For $\bar x\geq b$ and $x\leq \bar x$, $0\leq v^{(0)}(x,\bar x) \leq v^{(0)}(\bar x,\bar x)= R_\beta(\bar x)$. Since for $\eta=0$,
\begin{equation*}
  \lim_{x\to\infty} R_\beta(x) = 
  \frac\beta{1-\beta} 	\lim_{x\to\infty} \frac{ \int_x^\infty Z^{(q)}(y)^{-\frac1{1-\beta}}\mathrm d y }{ Z^{(q)}(x)^{-\frac1{1-\beta}} } 
  = \beta 	\lim_{x\to\infty} \frac{ Z^{(q)}(x) }{q W^{(q)}(x) } = \frac{\beta}{\Phi(q)}
\end{equation*}	
by l'H\^opital's rule in combination with \eqref{limit_ratio_scale} and since $v^{(0)}(\cdot,\cdot)$ is continuous on $D$, it follows that $v^{(0)}(\cdot,\cdot)$ is bounded on $D$. Let now $\eta>0$. Then $v^{(0)}(x,\bar x)-v^{(\eta)}(\bar x,\bar x)= \EE_{x,\bar{x}}
\left[ \eta \int_0^\infty e^{-qs} \, \dd K_s \right]$ and it is easy to see that the right hand side is decreasing in $x$ and decreasing in $\bar x$ from the definition of $(V,K)$ via the tax-reflection transform (recall Definition \ref{d:VK}). So for all $(x,\bar x)\in D$, $|v^{(\eta)}(x,\bar x)|\leq v^{(0)}(x,\bar x) + v^{(0)}(0,0)-v^{(\eta)}(0,0)$ and thus $v^{(\eta)}(\cdot,\cdot)$ is bounded on $D$.

Finally, we check that condition \ref{itemverif_derivbound} of Lemma \ref{l:verif} holds.
Recall $\overline W^{(q)}(x)=\int_0^x W^{(q)}(y)\mathrm d y$ for $x\geq 0$.
For $\bar x\geq b^*$ and $x\in[0,\bar x]$,
\begin{equation*}
  \begin{split}
    \partial_x v(x,\bar{x}) -\eta
    &= \eta (Z^{(q)}(x)-1)  +  R_\beta(\bar x)\frac{Z^{(q)\prime}(x)}{Z^{(q)}(\bar x)} \\
    & \leq  \eta q \overline W^{(q)}(x)  +   (1-\eta Z^{(q)}(\bar x)) \frac{Z^{(q)\prime}(x)}{Z^{(q)\prime}(\bar x)}  \\
    & \leq \eta q \left( \overline W^{(q)}(x)  + \overline W^{(q)}(\bar x)  \frac{W^{(q)}(x)}{W^{(q)}(\bar x)} \right) \\
    & \leq  0,
  \end{split}
\end{equation*}
where the first inequality is since $R_\beta(\bar x)\leq \frac{Z^{(q)}(\bar x)}{Z^{(q)\prime}(\bar x)} (1-\eta Z^{(q)}(\bar x))$ by Lemma \ref{lem_CQ}, the second inequality is due to $\eta\geq 1$ and the last inequality is due to  $x\leq \bar x$ and the (strict) log-concavity of $\overline W^{(q)}(\cdot)$ on $(0,\infty)$, where the latter follows from \eqref{avrametal} as this identity implies 
\begin{equation*}
  \overline W^{(q)}(a) \overline W^{(q)\prime\prime}(a)  - \overline W^{(q)\prime}(a)^2 = \frac1q \left( \mathbb E_{0,0} \left[ e^{-q \hat\tau_a} \right]-1 \right) \overline W^{(q)\prime\prime}(a)  <0, \quad a>0.
\end{equation*}
%which implies $\frac{\overline W^{(q)\prime}(a)}{\overline W^{(q)}(a)}$ is (strictly) decreasing for $a>0$.
%%%
On the other hand, for $\bar x\in[0,b^*)$ and $x\in[0,\bar x]$,
\begin{equation*}
  \begin{split}
    \partial_x v(x,\bar{x}) -\eta &=  \eta (Z^{(q)}(x)-1)  +  \left\{ R_\alpha(\bar x) + Z^{(q)}(\bar x)^{\frac1{1-\alpha}} C(b^*) \right\}   \frac{Z^{(q)\prime}(x)}{Z^{(q)}(\bar x)} \\
    &\leq  \eta (Z^{(q)}(x)-1)  +  \left\{ R_\alpha(b^*) + Z^{(q)}(b^*)^{\frac1{1-\alpha}} Q(b^*) \right\}   \frac{Z^{(q)\prime}(x)}{Z^{(q)}(b^*)} \\
    &=  \eta q \overline W^{(q)}(x)   +    (1-\eta Z^{(q)}(b^*))     \frac{Z^{(q)\prime}(x)}{Z^{(q)\prime}(b^*)} \\
    &\leq  \eta q \left( \overline W^{(q)}(x)  + \overline W^{(q)}(b^*) \frac{W^{(q)}(x)}{W^{(q)}(b^*)} \right) \\
    &\leq  0,
  \end{split}
\end{equation*}
where the last two inequalities follow by the same arguments as in the previous case and the first inequality follows because $C(b^*)\leq Q(b^*)$ by Lemma \ref{lem_CQ} and because, via the analogue of \eqref{Rgamderiv} for $R_\gamma$ and the definition of $Q(\cdot)$ in \eqref{def_CQ}, for $\bar x\in[0,b^*)$,
\begin{equation*}
  \begin{split}
    \frac{ R_\alpha(\bar x) + Z^{(q)}(\bar x)^{\frac1{1-\alpha}} C(b^*) }{Z^{(q)}(\bar x)} & = \frac{ R_\alpha(b^*) + Z^{(q)}(b^*)^{\frac1{1-\alpha}} C(b^*) }{Z^{(q)}(b^*)} \\
    & \quad {} - \int_{\bar x}^{b^*} \frac\alpha{1-\alpha} \frac{Z^{(q)\prime}(y)}{Z^{(q)}(y)^2} Z^{(q)}(y)^{\frac1{1-\alpha}} (C(b^*) - Q(y) )  \mathrm d y \\
    &\leq  \frac{ R_\alpha(b^*) + Z^{(q)}(b^*)^{\frac1{1-\alpha}} C(b^*) }{Z^{(q)}(b^*)},
  \end{split}
\end{equation*}
where the inequality is due to $C(b^*)\geq C(y)\geq Q(y)$ for $y\in[0,b^*)$, see Lemma \ref{lem_CQ}. 

%\begin{remark}
%  With $v(x,\bar x)$ given by the right hand side of \eqref{e:vtaxinj}, one can actually show that $\partial_x v(x,\bar{x})\leq \eta$ for $0\leq x\leq \bar x$ for any $b\geq 0$, not just for $b=b^*$. So $v$ satisfies conditions \ref{itemverif_formbelow0}, \ref{itemverif_genineq}, \ref{itemverif_bound} and \ref{itemverif_derivbound} in Lemma \ref{l:verif} for any $b\geq 0$ but condition \ref{itemverif_pdeineq} is satisfied only when $b=b^*$. 
%\end{remark}

\bibliographystyle{abbrvnat}
\bibliography{everything}

\end{document}